\documentclass[leqno,12pt]{amsart}
\usepackage[letterpaper,margin=1in]{geometry}
\usepackage{amsmath,amssymb,amsthm,mathtools}
\usepackage{microtype}
\usepackage[hidelinks]{hyperref}

\newtheorem{theorem}{Theorem}[section]
\newtheorem{lemma}[theorem]{Lemma}
\newtheorem{proposition}[theorem]{Proposition}
\newtheorem{corollary}[theorem]{Corollary}
\newtheorem{remark}[theorem]{Remark}
\newtheorem{example}[theorem]{Example}
\newtheorem{definition}[theorem]{Definition}

\newcommand{\R}{\mathbb R}
\newcommand{\Hh}{\mathcal H}
\newcommand{\Pp}{\mathcal P^+_{\lambda,\Lambda}}
\newcommand{\Pm}{\mathcal P^-_{\lambda,\Lambda}}
\newcommand{\Mplus}{\mathcal M^+_{\gamma,\Sigma,\mu_0}}
\newcommand{\Mminus}{\mathcal M^-_{\gamma,\Sigma,\mu_0}}

\newcommand{\tr}{\operatorname{tr}}

\makeatletter
\@namedef{subjclassname@2020}{\textup{2020} Mathematics Subject Classification}
\makeatother

\title[Equations with Unrectifiable Exceptional Gradient Sets]{Estimates on Fully Nonlinear Elliptic Equations with Unrectifiable Exceptional Gradient Sets}
\author[T. Phan]{Tuoc Phan}
\address[T. Phan]{Department of Mathematics, University of Tennessee, 227 Ayres Hall,
1403 Circle Drive, Knoxville, TN 37996-1320}
\email{tphan2@utk.edu}

\begin{document}
\maketitle
\vspace{-0.8cm}
\begin{abstract}
We study fully nonlinear uniformly elliptic equations that are imposed only when the gradients of solutions lie outside a closed ball and a closed purely one-unrectifiable set that can be unbounded. We prove Harnack inequalities and H\"{o}lder regularity estimates for viscosity solutions, with constants independent of the exceptional set. The proofs use the geometry of the contact sets in the cusp and barrier arguments to exclude exceptional gradients. We also provide examples and counterexamples showing the role of pure one-unrectifiability. Materials and proofs are presented in detail to ensure the paper is fully self-contained and accessible to a broad audience.
\end{abstract}
\section{Introduction and main results}
\subsection{Problem setting and main results} We study fully nonlinear equations that are imposed only when the gradients of solutions are outside a subset set of $\R^n$ that can be unbounded. To set up the problem, let \(\gamma\ge0\) and let \(\Sigma\subset\R^n\) be a closed set. The considered class of equations is imposed only when 
\[
 Du \not\in (\overline{B}_\gamma \cup \Sigma),
\]
where $u$ is an unknown viscosity solution, and $B_\gamma$ is the ball in $\R^n$ of radius $\gamma$ centered at the original. Thus the excluded gradient set is \(\overline B_\gamma\cup\Sigma\), and main point is that \(\Sigma\) can be unbounded.  This work provides sufficient geometric conditions on $\Sigma$ under which the Harnack and H\"{o}lder regularity estimates are proved for viscosity solutions. Note that if \(\Sigma\) is bounded, then the set  \(\overline B_\gamma\cup\Sigma\) is contained in a larger ball. In this case, the gradient threshold may be increased accordingly and the considered equations are reduced to the ones already studied in the literature such as \cite{ImbertSilvestre, Mooney}. 

\smallskip
To formulate the problem, let us denote $\mathbb S^n$ the set of $n\times n$ symmetric matrices. For given numbers \(0<\lambda<\Lambda\) and \(\mu_0\ge0\), we consider a fully nonlinear operator
\[
 F:\mathbb S^n\times\R^n\times B_1\to\R
\]
which is continuous and satisfies
\begin{equation} \label{F-cond}
 \Pm(X-Y)-\mu_0|p-q|\le F(X,p,x)-F(Y,q,x)\le \Pp(X-Y)+\mu_0|p-q|
\end{equation}
for all $X,Y\in\mathbb S^n$, $p,q\in\R^n$, and $x\in B_1$. Here, $\Pm$ and $\Pp$ respectively denote the minimal and maximal Pucci operators with the corresponding given ellipticity constants $\lambda$ and $\Lambda$. Note that no convexity or differentiability of \(F\) is assumed. We study the equation
\begin{equation} \label{intro-PDE}
 F(D^2u,Du,x)=f(x) \quad \text{in} \quad B_1 \quad \text{when} \quad Du \not\in (\overline{B}_\gamma \cup \Sigma),
\end{equation}
where $\Sigma \subset \R^n$ is a given closed set. Solutions of \eqref{intro-PDE} are understood in the viscosity sense as in Definition~\ref{def:restricted-viscosity} below.

\smallskip
To put the study in perspectives, let us point out several related results in the research line.  Caffarelli and Salazar studied fully nonlinear elliptic equations with patches of zero gradient \cite{CaffarelliSalazar02} in which the existence, Harnack inequalities, and H\"{o}lder estimates are proved. Imbert and Silvestre proved H\"{o}lder estimates and Harnack inequalities for fully nonlinear uniformly elliptic equations imposed only when $|Du|$ exceeds a fixed threshold \cite{ImbertSilvestre}. Mooney obtained a Harnack inequality for the same class of equations with unbounded drifts in \cite{Mooney}. Chang-Lara and Pimentel studied the Hamilton-Jacobi equation
$\min\{-\Delta u-1,|Du|-1\}=0$, which imposes a lower bound on the
gradient norm \cite{Chang}. They proved local Lipschitz
continuity of solutions and continuity of the gradient norm. Ara\'ujo, Sobral, and Teixeira in \cite{AraujoSobralTeixeira25} studied gradient-activated equations of the form $(|Du|-\kappa)_+^q\mathcal F(D^2u)=f$ in $\{|Du|>\kappa\}$, and obtaining gradient bounds and continuity of $Du$ up to the free boundary.  See also \cite{Chang-1} for further discussion and related open problems in this research direction. The gradient restrictions in these equations concern zero or small gradients. Our primary focus in this paper is to demonstrate that results on Harnack inequalities and H\"{o}lder estimates still hold  even  the equations \eqref{intro-PDE} are not imposed when \(Du \in \Sigma\) with \(\vert{}Du\vert{}\) is arbitrarily large.

\smallskip
Aside from the mathematical research \cite{AraujoSobralTeixeira25, CaffarelliSalazar02, Chang, ImbertSilvestre, Mooney} that we discussed above, equations with gradient restrictions occur in models from other fields such as porous flow, superconductivity, and optimal control theory. Pascal studied non-steady flow in a one-dimensional porous reservoir with a threshold pressure gradient \cite{Pascal81}. A simple threshold law for the pressure $p$ and flux $q$ is
\[
 q=-k(|p_x|-\gamma)_+\operatorname{sgn}(p_x),\qquad
 q_x=-k p_{xx}\quad\text{when }|p_x|>\gamma,
\]
where $k>0$ and $\gamma\ge0$ are constants. Hence the stationary balance $q_x=0$ gives $p_{xx}=0$ when $|p_x|>\gamma$. Nash and Rees studied the Buckingham-Reiner models for the one-dimensional flow of a Bingham fluid
along a uniform pipe or channel with more
general one-dimensional configurations and they showed how the geometry of the pores changes the relation between flux and the threshold pressure gradient \cite{NashRees17}. In a simplified stationary model for superconducting vortices, Caffarelli, Salazar, and Shahgholian studied
\[
 \Delta h=h\chi_{\{|Dh|>0\}},\qquad h\ge0,
\]
which gives $\Delta h=h$ where $|Dh|>0$ \cite{CaffarelliSalazarShahgholian04}. Here the right-hand side depends on $h$. The Bean critical-state model has the gradient constraint $|Dh|\le J_c$ in the cylindrical reduction \cite{Prigozhin96}. Note also that the considered equations \eqref{intro-PDE} are closely related to the Hamilton-Jacobi-Bellman equations studied in \cite{Dammann, Evans-1, Hynd}, for example,
\begin{equation} \label{HJB-oct-7}
\max\left\{
 F(D^2u, Du, u, x)-f(x),\,
H(Du)\right\} = 0 \quad \text{in} \quad \Omega \subset \R^n,
\end{equation}
for some given open domain $\Omega$. Observe that \eqref{HJB-oct-7} gives the equation $F(D^2u, Du, u, x)-f(x) =0$ in $\Omega$ when $Du \not\in \{ p \in \R^n: H(p) \geq 0\}$. These mentioned models motivate equations imposed in an active gradient region. The additional exclusion of $\Sigma$ in \eqref{intro-PDE} is our geometric assumption.

\smallskip
Now, we explain and introduce the assumption that we impose on $\Sigma$. We note that a condition for $\Sigma$ based only on Lebesgue measure or density cannot be sufficient. For example, for \(n\ge2\), consider \(\Sigma=\R e_1\), then \(\Sigma\) has zero \(n\)-dimensional measure, but every function depending only on \(x_1\) has its gradient in \(\Sigma\), so the equation may be inactive everywhere. The same obstruction occurs for a ray. The relevant issue is therefore whether the exceptional set contains one-dimensional channels that can carry the gradient image. Our hypothesis is that \(\Sigma\) is purely one-unrectifiable, that is
\[
 \Hh^1(\Sigma\cap\Gamma)=0
\]
for every Lipschitz curve \(\Gamma\subset\R^n\), where $ \Hh^1$ denotes the \(1\)-dimensional Hausdorff measure.  The precise definition and discussion on the purely one-unrectifiable sets and their properties are given in Section \ref{sec:geometry} below.

\smallskip
We emphasize that the purely one-unrectifiable concept is classical in geometric measure theory. For sets of finite length, the decomposition into rectifiable and purely unrectifiable parts and the Besicovitch-Federer projection theorem are basic structure results; see \cite{Federer, Mattila}. Quantitative rectifiability is closely related to singular integrals and harmonic analysis; see David and Semmes \cite{DavidSemmes}. In the plane, pure unrectifiability also occurs in analytic capacity, the Cauchy transform, and the Painlev\'e problem; see Tolsa \cite{Tolsa}. The four-corner Cantor set is the standard example of a purely one-unrectifiable set with positive and finite one-dimensional Hausdorff measure. See Remark \ref{remark-oct-7} and Section \ref{Example-Sec} for more examples on purely one-unrectifiable sets. On the other hand, we also point out that the condition $\dim_{\mathcal{H}}E=1$  alone does not imply that \(E\) is purely one-unrectifiable.  Indeed, it follows from the decomposition theorem for 1-sets (\cite[Theorem 15.6, p. 205]{Mattila} that for any Borel set $E \subset \R^n$ satisfying $\Hh^1(E) \in (0, \infty)$,  $E$ admits a pairwise-disjoint decomposition
\[
E= E_r \cup E_u \cup E_0
\]
where $E_u$ is purely one-unrectifiable, $E_r$ is one-rectifiable, and $\Hh^1(E_0) =0$. Moreover, after modifying the exact choice of the decomposition, we may assume that both the
sets $E_r, E_u \subset \R^n$ are Borel sets. See  \cite{Aliaga, Bongers} for some examples and details on geometric properties and applications of purely one-unrectifiable sets.

\smallskip
Our first main result is the following Harnack inequalities for viscosity solutions to \eqref{intro-PDE}. 
\begin{theorem}[Harnack inequalities]\label{thm:harnack-intro} For every $0<\lambda<\Lambda$, and $\mu_0\geq 0$, there is $C = C(n,\lambda,\Lambda,\mu_0)>0$ such that the following assertion holds. Let $\Sigma\subset\R^n$ be closed and purely one-unrectifiable,  $F$ satisfy the structural condition \eqref{F-cond} and $F(0, 0,\cdot) \in L^\infty(B_1)$, and $f\in C(B_1)\cap L^\infty(B_1)$. Suppose that  $u\in C(B_1)$ is a non-negative viscosity solution of \eqref{intro-PDE}. Then
\[
 \sup_{B_{1/2}}u\le C\left(\inf_{B_{1/2}}u+ \|f\|_{L^\infty(B_1)}+\|F(0,0,\cdot)\|_{L^\infty(B_1)} +\gamma\right).
\]
\end{theorem}\noindent
In this paper, we also prove the following result on the H\"{o}lder regularity estimates for \eqref{intro-PDE}, which is a direct consequence of Theorem~\ref{thm:harnack-intro}.
\begin{theorem}[H\"{o}lder estimates]\label{thm:holder-intro} For every  $0<\lambda<\Lambda$,  and $\mu_0\geq 0$, there are $\alpha=\alpha(n,\lambda,\Lambda,\mu_0)\in(0,1)$ and $C=C(n,\lambda,\Lambda,\mu_0)>0$ such that the following assertions hold.
 Let $\Sigma\subset\R^n$ be closed and purely one-unrectifiable,  $F$ satisfy  the structural condition \eqref{F-cond} and $F(0, 0,\cdot) \in L^\infty(B_1)$, and $f\in C(B_1)\cap L^\infty(B_1)$. Suppose that $u\in C(B_1)\cap L^\infty(B_1)$ is a viscosity solution of \eqref{intro-PDE} in $B_1$. Then $u \in C^\alpha(B_{1/2})$ and 
\[
 [u]_{C^\alpha(B_{1/2})}\le C\left(K+K^{1-\alpha}\gamma^\alpha\right),
\]
where
\[
K=\|u\|_{L^\infty(B_1)}+\|f\|_{L^\infty(B_1)}+\|F(0,0,\cdot)\|_{L^\infty(B_1)}.
\]
\end{theorem}
\smallskip
\noindent
We note that Theorems \ref{thm:harnack-intro}-\ref{thm:holder-intro} apply without convexity or concavity assumption on $F$. For example,
\[
 F(X,p,x)=\tr(A(x)X)+G(X,x)+b(x)\cdot p,
\]
where \(A\) and \(b\) are continuous, \(\lambda I\le A(x)\le\Lambda I\), \(b\) is bounded, and \(G\) is any continuous uniformly elliptic perturbation for which \eqref{F-cond} holds. They also apply to Bellman and Isaacs operators
\[
 F(X,p,x)=\inf_\alpha\sup_\beta\{\tr(A_{\alpha\beta}(x)X)+b_{\alpha\beta}(x)\cdot p+c_{\alpha\beta}(x)\},
\]
provided the coefficient families are equicontinuous in $x$ on compact subsets, $A_{\alpha\beta}$ satisfies the standard ellipticity and boundedness conditions, and $b_{\alpha\beta}, c_{\alpha\beta}$ are uniformly bounded. The PDE associated with these operators are not  required when the solution gradients are in \(\overline B_\gamma\cup\Sigma\). It is an interesting problem to see if Theorems \ref{thm:harnack-intro}-\ref{thm:holder-intro} can be extended to the case with singular drifts as in \cite{Mooney, Safonov-1}, and we plan to investigate this problem in the near future.
\subsection{Difficulties and approaches} To prove Theorem \ref{thm:harnack-intro}, we follow and adjust the approach introduced in \cite{ImbertSilvestre}, see also \cite{Krylov, Safonov, Safonov-1} and the recent work \cite{CFP} for the Krylov-Safonov approach. There are two main new difficulties compared with \cite{ImbertSilvestre} due to $\Sigma$ which is allowed to be unbounded. The first occurs in the cusp argument. For the cusp \(\phi(x)=-10|x|^{1/2}\), the original proof chooses the gradient threshold so that every cusp contact has \(|D\phi|>\gamma\). In our setting this is no longer enough, because the contact gradient may belong to \(\Sigma\), even when it is arbitrarily large. At a cusp contact, however, \(D^2u\ge D^2\phi\), and the radial eigenvalue of \(D^2\phi\) is positive. Hence the derivative of the gradient map has rank at least one at almost every twice differentiable contact point. We prove that a Lipschitz gradient map with this property cannot take values in a purely one-unrectifiable set on a set of positive measure. The proof is elementary: after a Lipschitz extension, one uses density points, Fubini's theorem, and the one-dimensional area formula. It follows that the equation is active at almost every cusp contact point, and the Jacobian estimate of the cusp map can then be recovered.

\smallskip
The second difficulty occurs in the doubling argument. In \cite{ImbertSilvestre}, the radial barrier is chosen so that \(|DB|>\gamma\), and direct comparison applies. Here \(|DB|>\gamma\) does not exclude \(DB\in\Sigma\). We therefore perturb the barrier by small affine functions. The corresponding tilted contact set has positive measure; on this set the contact-gradient map is Lipschitz and its derivative again has rank at least one. The same exclusion lemma yields an active contact point, at which the strict Pucci inequality gives the contradiction. This tilted-barrier argument is needed only because the exceptional set may occur at large gradients.

\smallskip
Note also that compared to \cite{ImbertSilvestre}, the structure and organization of our proofs of main results are different. After the measure estimate and doubling property, we obtain the \(L^\varepsilon\)-estimate as in \cite{ImbertSilvestre}. We then prove Harnack's inequality first. H\"{o}lder continuity is derived afterward from Harnack by applying it to the two non-negative functions
\[
 u-m_R+\frac{R\gamma}{2},\qquad M_R-u+\frac{R\gamma}{2},
\]
where \(M_R\) and \(m_R\) are the supremum and infimum of the solution $u$ on a ball of radius \(R\). The additive term \(R\gamma/2\) keeps the normalized gradient threshold uniformly bounded, and the two Harnack inequalities give a quantitative decay of the oscillation. This yields the H\"{o}lder estimate directly from Harnack. 
\subsection{Paper organization} The remaining parts of the paper is organized as follows. Section~\ref{Prelim-Sec} gives various definitions and estimates needed for the paper. These include the viscosity formulation, the reduction to the extremal inequalities, the geometric contact-gradient lemma. Section~\ref{Cusp-sec} contains the cusp measure estimate. The tilted-barrier comparison and doubling property are presented in Section \ref{doubling-property}. Section~\ref{L-eps-Sec} is about the $L^\varepsilon$-estimates. The proofs of  Theorem~\ref{thm:harnack-intro} and Theorem~\ref{thm:holder-intro} are then provided in Section~\ref{Proof-Sec}. In the last section, Section~\ref{Example-Sec}, we provide examples and counterexamples on our assumption for $\Sigma$.
\section{Definitions and Preliminary Estimates} \label{Prelim-Sec}
\subsection{Pucci operators and extremal reduction}\label{sec:fully-nonlinear} Let $0 < \lambda < \Lambda$ be given numbers. Then, the maximal and minimal Pucci operators with the corresponding ellipticity constants $\lambda$ and $\Lambda$ are defined as
\[
 \Pp(X)=\Lambda\sum_{e_j>0} e_j  + \lambda\sum_{e_j<0} e_j,\quad \text{and}  \quad \Pm(X)=\lambda \sum_{e_j>0} e_j+ \Lambda \sum_{e_j<0} e_j,
\]
for $X\in\mathbb S^n$, where $\{e_j\}_{j=1}^n$ are the eigenvalues of $X$. We assume that $F=F(X,p,x)$ is continuous and it satisfies
\begin{equation}\label{F-structure}
 \Pm(X-Y)-\mu_0|p-q|\le F(X,p,x)-F(Y,q,x)\le \Pp(X-Y)+\mu_0|p-q|
\end{equation}
for every $X,Y\in\mathbb S^n$, $p,q\in\R^n$, and for every $x$ in the corresponding considered domain. 
\begin{definition} \label{def:restricted-viscosity}
Let $\Omega\subset\R^n$ be open, $\gamma\ge0$, and $\Sigma\subset\R^n$ be closed. A function $u\in C(\Omega)$ is a viscosity solution of
\begin{equation} \label{main-eqn}
 F(D^2u,Du,x)=f(x) \quad \text{in} \quad \Omega \quad \text{when} \quad Du \not\in (\overline{B}_\gamma \cup \Sigma)
\end{equation}
with some given $f\in C(\Omega)$ if the following conditions hold. If $\phi\in C^2(\Omega)$ touches $u$ from below at $x_0\in\Omega$ with $|D\phi(x_0)|>\gamma$, $D\phi(x_0)\notin\Sigma$, then
\[
 F(D^2\phi(x_0),D\phi(x_0),x_0)\le f(x_0).
\]
If $\phi$ touches $u$ from above at $x_0$ and the same gradient condition holds, then
\[
 F(D^2\phi(x_0),D\phi(x_0),x_0)\ge f(x_0).
\]
\end{definition}
\noindent
We note that the definition agrees with the usual viscosity definition except that we impose no equation when $|D\phi(x_0)|\le\gamma$ or $D\phi(x_0)\in\Sigma$.

\smallskip
Following the convention used in \cite{ImbertSilvestre}, we define the extremal operators
\[
 \Mplus(X,p)=\begin{cases}\Pp(X)+\mu_0|p|,&\  p\not \in(\overline{B}_\gamma \cup \Sigma),\\ +\infty,&\  p\in (\overline{B}_\gamma \cup \Sigma),\end{cases}
\]
and
\[
 \Mminus(X,p)=\begin{cases}\Pm(X)-\mu_0|p|,& \  p\not \in(\overline{B}_\gamma \cup \Sigma) ,\\ -\infty,&\  p \in(\overline{B}_\gamma \cup \Sigma),\end{cases}
\]
for $(X, p) \in \mathbb{S}^n \times \R^n$. As in Definition \ref{def:restricted-viscosity}, the  inequalities involving $\Mminus(D^2u, Du)$ with lower semi-continuous function $u$ (and respectively and $\Mplus(D^2u, Du)$  with upper semi-continuous function $u$) are understood in the viscosity sense.

\begin{proposition} \label{prop:extremal-reduction} Let $F$ satisfy \eqref{F-structure}, $u \in C(\Omega)$ be a viscosity solution of \eqref{main-eqn} with $f \in C(\Omega) \cap L^\infty(\Omega)$. Assume that 
\[
 C_F=\|f\|_{L^\infty(\Omega)}+\|F(0,0,\cdot)\|_{L^\infty(\Omega)} < \infty.
\]
Then, in the viscosity sense,
\[
 \Mminus(D^2u,Du)\le C_F, \quad \text{and} \quad \Mplus(D^2u,Du)\ge-C_F \quad \text{in} \quad \Omega.
\]
\end{proposition}

\begin{proof} The proof is standard, see \cite{CaffarelliCabre} for example. However, as it is not long and for completeness, we provide the details. Let $\phi \in C^2$ touch $u$ from below at $x_0$ and assume $p=D\phi(x_0) \not \in (\overline{B}_\gamma \cup \Sigma)$. With $X=D^2\phi(x_0)$, Definition~\ref{def:restricted-viscosity} and \eqref{F-structure} give
\[
 \Pm(X)-\mu_0|p|\le F(X,p,x_0)-F(0,0,x_0)\le f(x_0)-F(0,0,x_0)\le C_F.
\]
This proves the lower extremal inequality. If $\phi$ touches from above at an active gradient, then
\[
 \Pp(X)+\mu_0|p|\ge F(X,p,x_0)-F(0,0,x_0)\ge f(x_0)-F(0,0,x_0)\ge-C_F,
\]
which proves the upper inequality.
\end{proof}
\subsection{Purely one-unrectifiable sets and contact gradients}\label{sec:geometry}
We begin with recalling the definition of \(s\)-dimensional Hausdorff measure. Let \(s\geq 0\), \(\delta>0\), and \(E\subset \mathbb{R}^n\). The
\(s\)-dimensional Hausdorff content at the scale \(\delta\) is defined by
\[
\mathcal{H}^s_\delta(E)=
\inf\Big\{
\sum_{j=1}^\infty
\bigl(\operatorname{diam} U_j\bigr)^s: E\subset \bigcup_{j=1}^\infty U_j,
\quad
\operatorname{diam} U_j\leq \delta
\Big\}.
\]
Then, the \(s\)-dimensional Hausdorff measure of \(E\) is defined by
\[
\mathcal{H}^s(E)
=
\lim_{\delta\downarrow 0}\mathcal{H}^s_\delta(E)
=
\sup_{\delta>0}\mathcal{H}^s_\delta(E).
\]
Note that a dimensional normalization constant is sometimes included in the definition of $\mathcal{H}^s_\delta$, see \cite[Definition 2.1, p. 81]{Evans}. Such a normalization does not affect Hausdorff dimension or the notion of \(\mathcal{H}^1\)-null sets. 
Also, it follows directly from the definition that
\[
\mathcal{H}^1(E)
=
\lim_{\delta\downarrow 0}
\inf\Big\{
\sum_{j=1}^\infty \operatorname{diam} U_j
:
E\subset \bigcup_{j=1}^\infty U_j,
\quad
\operatorname{diam} U_j\leq \delta
\Big\}.
\]
Recall also that the Hausdorff dimension of a measurable set \(E \subset \R^n\) is defined by
\[
\dim_{\mathcal{H}} E =
\inf\left\{
s\geq 0:\mathcal{H}^s(E)=0
\right\}.
\]
It then follows from the definition that if $\dim_{\mathcal{H}}E=d$, 
then
\[
\mathcal{H}^s(E)=\infty
\quad\text{for }s<d,
\quad \text{and} \quad \mathcal{H}^s(E)=0
\quad\text{for }s>d.
\]
At the critical exponent \(s=d\), the quantity
\(\mathcal{H}^d(E)\) may be zero, finite and positive, or infinite. Interested readers can find in \cite[Chapter 2]{Evans} for more details on the \(s\)-dimensional Hausdorff measure.

\smallskip
We provide the following definition of the purely one-unrectifiable sets. 
\begin{definition} A set \(\Sigma\subset\R^n\) is \emph{purely one-unrectifiable} if
\[
 \Hh^1(\Sigma\cap g(I))=0
\]
for every bounded interval \(I\subset\R\) and every Lipschitz map \(g:I\to\R^n\).
\end{definition}
\noindent
\smallskip
The following elementary facts will be used repeatedly in the paper.
\begin{lemma}\label{lem:stability-sigma}
The following statements hold.
\begin{enumerate}
\item If \(\Hh^1(\Sigma)=0\), then \(\Sigma\) is purely one-unrectifiable.
\item A countable union of purely one-unrectifiable sets is purely one-unrectifiable.
\item If \(\Sigma\) is purely one-unrectifiable, then \(a\Sigma+b\) is purely one-unrectifiable for every \(a\ne0\) and \(b\in\R^n\). In particular, \(-\Sigma\) is purely one-unrectifiable.
\end{enumerate}
\end{lemma}

\begin{proof}
The first statement is immediate. For the second, if \(\Sigma=\bigcup_j\Sigma_j\) and \(g\) is Lipschitz, then
\[
 \Hh^1(\Sigma\cap g(I))\le\sum_j\Hh^1(\Sigma_j\cap g(I))=0.
\]
For the third, if \(S(x)=ax+b\), then \(S^{-1}\circ g\) is Lipschitz and
\[
 \Hh^1\bigl((a\Sigma+b)\cap g(I)\bigr)=|a|\Hh^1\bigl(\Sigma\cap S^{-1}(g(I))\bigr)=0.
\]
\end{proof}
\begin{remark} \label{remark-oct-7}
It follows from Lemma \ref{lem:stability-sigma}-\textup{(i)} that if \(\dim_{\Hh}\Sigma<1\), then \(\Hh^1(\Sigma)=0\), hence \(\Sigma\) is purely one-unrectifiable. On the other hand, note also that the class of purely one-unrectifiable sets is strictly larger. Indeed, in \(\R^2\), let  \(\mathcal K\) be the four-corner Cantor set defined by
\begin{equation} \label{four-corner-Cantor}
 \mathcal K=C_{1/4}\times C_{1/4} \quad \text{where} \quad C_{1/4}=\Big\{\sum_{j=1}^\infty a_j4^{-j}:a_j\in\{0,3\}\Big\}.
\end{equation}
As indicated in \cite[Example 15.2 and the discussion following
Definition 15.3, pp.~203-204]{Mattila}, \(\Hh^1(\mathcal K) \in (0, \infty)\), and \(\mathcal K\) is purely one-unrectifiable. Thus pure one-unrectifiability is not a small Hausdorff-dimension assumption.
\end{remark}
\subsection{One-dimensional area formula and a contact-gradient exclusion lemma}
We recall the following form of the one-dimensional area formula that is needed in the paper. Let $I \subset \R$ be an interval, \(g:I\to\R^n\) be Lipschitz, and \(E\subset I\) be measurable. It follows from area formula (see \cite[Theorem 3.8, p.119]{Evans} or \cite[Theorem 3.2.3, p. 243]{Federer}) that
\begin{equation} \label{one-dim-area}
 \int_E |g'(s)|\,ds=\int_{g(E)}N(g,E,z)\,d\Hh^1(z),
\end{equation}
where \(N(g,E,z)=\#\{s\in E:g(s)=z\}\). In particular, if \(\Hh^1(g(E))=0\), then \(g'=0\) almost everywhere on \(E\). 

\smallskip
For a measurable set \(E\subset\R^n\), we say that a map \(G:E\to\R^n\) is differentiable relative to \(E\) at \(x\in E\) if there is an $n\times n$ derivative matrix \(A_x\) such that
\[
 G(y)=G(x)+A_x(y-x)+o(|y-x|)\qquad\text{as }y\to x,\ y\in E.
\]
We now state and prove the following lemma on the contact-gradient exclusion in purely one-unrectifiable set.
\begin{lemma}\label{lem:contact-gradient}
Let \(E\subset\R^n\) be measurable,  \(\Sigma\subset\R^n\) be closed and purely one-unrectifiable, and let \(G:E\to\R^n\) be Lipschitz. Suppose that \(G\) is differentiable relative to \(E\) at almost every point of
\[
 E_\Sigma=\{x\in E:G(x)\in\Sigma\}
\]
and the derivative \(A_x\) of $G$ satisfies \(\operatorname{rank}A_x\ge1\) almost everywhere \(x \in E_\Sigma\). Then \(|E_\Sigma|=0\).
\end{lemma}

\begin{proof} We assume by contradiction that \(|E_\Sigma|>0\). From the continuity of the Lebesgue measure, and replacing \(E_\Sigma\) by its intersection with a sufficiently large ball, we may assume that \(E_\Sigma\)  is bounded and still has positive measure. Using Kirszbraun's theorem (see \cite{Kirszbraun, Valentine}), we can extend the function \(G\) Lipschitzly to \(\R^n\), and denote the resulting map by \(\widetilde G\). By Rademacher's theorem (see \cite{Rademacher} and \cite[Theorem 3.2, p. 103]{Evans}), \(\widetilde G\) is differentiable almost everywhere. At almost every Lebesgue density point \(x\) of \(E_\Sigma\) where the relative derivative \(A_x\) and \(D\widetilde G(x)\) both exist, one has
\[
 D\widetilde G(x)=A_x.
\]
To prove this claim, let us fix \(e\in\mathbb S^{n-1}\) and we prove that \(D\widetilde G(x)e=A_xe\). Since \(E_\Sigma\) has density one at \(x\), we first claim that
\begin{equation} \label{claim-oct-5}
 \operatorname{dist}(x+he,E_\Sigma)=o(|h|)\qquad\text{as }h\to0.
\end{equation}
Indeed, if \eqref{claim-oct-5} fails, there are $\delta>0$ and $h_k\to0$ when $k \rightarrow \infty$ such that
\[
 \operatorname{dist}(x+h_ke,E_\Sigma)\geq\delta|h_k|, \quad \forall \ k \in \mathbb{N}.
\]
Hence
\[
 B_{\delta|h_k|}(x+h_ke)
 \subset B_{(1+\delta)|h_k|}(x)\setminus E_\Sigma,
\]
and therefore
\[
 \frac{|B_{(1+\delta)|h_k|}(x)\setminus E_\Sigma|}
      {|B_{(1+\delta)|h_k|}|}
 \geq \left(\frac{\delta}{1+\delta}\right)^n.
\]
This contradicts the density-one property of $E_\Sigma$ at $x$.


\smallskip
Now, due to \eqref{claim-oct-5}, we can choose $y_h\in E_\Sigma$ such that
\[
 |y_h-(x+he)|=o(|h|).
\]
As $\widetilde G=G$ on $E_\Sigma$ and $\widetilde G$ is Lipschitz, it follows that
\[
 \widetilde G(x+he)-\widetilde G(x)
 =G(y_h)-G(x)+O(|x+he-y_h|)
 =G(y_h)-G(x)+o(|h|).
\]
Also, the differentiable relative expansion of $G$ at $x$ gives
\[
 G(y_h)-G(x)=A_x(y_h-x)+o(|y_h-x|).
\]
As $y_h-x=he+o(|h|)$, we obtain
\[
 \widetilde G(x+he)-\widetilde G(x)
 =hA_xe+o(|h|).
\]
Hence \(D\widetilde G(x)e=A_xe\), and therefore \(D\widetilde G(x)=A_x\).

\smallskip
Since $\operatorname{rank}D\widetilde G(x)\geq1$ for almost every
$x\in E_\Sigma$, at least one column of $D\widetilde G(x)$ is nonzero.
Consequently, up to a null set,
\[
 E_\Sigma
 =\bigcup_{i=1}^{n}\bigcup_{k=1}^{\infty}
 \{x\in E_\Sigma:|D\widetilde G(x)e_i|\geq k^{-1}\}.
\]
As $|E_\Sigma|>0$, countable subadditivity yields indices
$i\in\{1,\ldots,n\}$ and $k\geq1$ for which
\[
 \mathcal{F}=\{x\in E_\Sigma:|D\widetilde G(x)e_i|\geq k^{-1}\}
\]
has positive measure. Now, for each $x \in \R^n$, write $x=z+se_i$, where $z\in e_i^\perp$ and $s\in\mathbb R$. Since
this change of coordinates has Jacobian one, it follows from Tonelli's theorem that
\[
 0<|F|
 =\int_{e_i^\perp}\int_{\mathbb R}
    \mathbf 1_F(z+se_i)\,ds\,dz
 =\int_{e_i^\perp}|F_z|\,dz,
\]
where $dz$ denotes Lebesgue measure on $e_i^\perp$ and $F_z=\{s\in\mathbb R:z+se_i\in \mathcal{F}\}$. For $n=1$, this is the counting measure on $e_i^\perp=\{0\}$. Hence $|F_z|>0$ for some $z\in e_i^\perp$. We then let \(g(s)=\widetilde G(z+se_i)\). Since \(g\) is Lipschitz and, for almost every \(s\in F_z\),
\begin{equation} \label{g-deriv-oct-5}
 |g'(s)|=|D\widetilde G(z+se_i)e_i|\ge k^{-1}.
\end{equation}
On the other hand, \(g(F_z)\subset\Sigma\cap g(I)\) for a bounded interval \(I\) containing \(F_z\). Since \(\Sigma\) is purely one-unrectifiable, \(\Hh^1(g(F_z))=0\). From this, and the one-dimensional area formula \eqref{one-dim-area}, we see that \(g'=0\) almost everywhere on \(F_z\), and we obtain a contradiction to \eqref{g-deriv-oct-5}. The proof is then complete.
\end{proof}

\subsection{Semi-concave functions} \label{semi-concave}

We record the properties of semi-concave functions used below. Here semi-concavity means local semi-concavity with a quadratic modulus: on each ball compactly contained in the domain, $u(z)-K|z|^2/2$ is concave for some $K\ge0$. We denote by $D^+u(x)$ the superdifferential, that is, the set of $p\in\R^n$ such that
\[
 u(x+h)\le u(x)+p\cdot h+o(|h|)\qquad\text{as }h\to0.
\]
The
second-order expansion is understood as in
\cite[Definition~1.4, p.~9]{CaffarelliCabre}; see also
\cite[Theorem~1.5, p.~9]{CaffarelliCabre} for its almost-everywhere
existence for convex functions. The following known results are needed in the paper, we also provide the proof for completeness.
\begin{lemma}\label{lem:semiconcave-properties}
Let \(u\) be semi-concave in an open set \(\Omega\subset\R^n\).
Then \(u\) is locally Lipschitz and has a second-order expansion almost
everywhere in \(\Omega\). Suppose that, at \(x\in\Omega\),
\begin{equation}\label{eq:semiconcave-expansion}
 u(x+h)=u(x)+p\cdot h+\frac12h^TAh+o(|h|^2),
 \qquad A\in\mathbb S^n.
\end{equation}
Then \(p=Du(x)\), and for every sequence \(x_j\to x\) and every choice
\(p_j\in D^+u(x_j)\),
\begin{equation}\label{eq:supergradient-expansion}
 p_j=p+A(x_j-x)+o(|x_j-x|).
\end{equation}
In addition, if \(\phi\in C^1\) touches \(u\) from below at \(x\), then \(u\) is
differentiable at \(x\) and \(Du(x)=D\phi(x)\). Moreover, suppose that \eqref{eq:semiconcave-expansion} holds at \(x\),
that \(u\) sloves 
\[
 \Mminus(D^2u,Du)\le c\quad\text{in }\Omega,
\]
in the viscosity sense, and \(|p|>\gamma\), \(p\notin\Sigma\). Then
\begin{equation}\label{eq:semiconcave-classical-inequality}
 \Pm(A)-\mu_0|p|\le c.
\end{equation}
\end{lemma}

\begin{proof} The assertion that $u$ is locally Lipschitz follows from \cite[Theorem 2.17, p. 33]{CannarsaSinestrari}. On the other hand, on each ball compactly contained in \(\Omega\), there is \(K\ge0\)
such that \(K|z|^2/2-u(z)\) is convex in this ball. Then, the almost-everywhere expansion follow from the Aleksandrov's theorem for convex functions; see
\cite[Theorem~2.3.1(i), p. 42]{CannarsaSinestrari} and \cite[Theorem 6.9, p. 273]{Evans}.

\smallskip
Fix \(x\) satisfying \eqref{eq:semiconcave-expansion}, and choose such a
constant \(K\) such that the function $\psi(z)= K|z|^2/2-u(z)$ is convex on a small ball centered at \(x\). Let us define
\[
 P(z)=u(x)+p\cdot(z-x)+\frac12(z-x)^TA(z-x),
 \qquad R(z)=u(z)-P(z).
\]
Thus \(R(z)=o(|z-x|^2)\). It suffices to consider indices $j$ for which
\(r_j=|x_j-x|>0\), since \(D^+u(x)=\{p\}\). Let \(h_j=x_j-x\), \(0<\delta<1\), \(|e|=1\), and \(t_j=\delta r_j\). For all
such sufficiently large \(j\), it follows from the convexity of $\psi$ that
\[
 u(x_j\pm t_je)
 \le u(x_j)\pm t_jp_j\cdot e+\frac K2t_j^2;
\]
see \cite[Proposition~3.3.1, pp. 56-57]{CannarsaSinestrari}, for example. Hence
\[
 \frac{u(x_j+t_je)-u(x_j)}{t_j}-\frac K2t_j
 \le p_j\cdot e
 \le
 \frac{u(x_j)-u(x_j-t_je)}{t_j}+\frac K2t_j.
\]
Since $u = P +R$, and
\[
 \begin{aligned}
 \frac{P(x_j+t_je)-P(x_j)}{t_j}
   &=p\cdot e+(Ah_j)\cdot e+\frac{t_j}{2}e^TAe,\\
 \frac{P(x_j)-P(x_j-t_je)}{t_j}
   &=p\cdot e+(Ah_j)\cdot e-\frac{t_j}{2}e^TAe,
 \end{aligned}
\]
we obtain
\begin{align*}
 \left|(p_j-p-Ah_j)\cdot e\right|
 & \le \frac{K+\|A\|_{\mathrm{op}}}{2}\delta r_j
      +\frac{2}{\delta r_j}
       \sup_{z\in B_{2r_j}(x)}|R(z)|\\
       & \leq \frac{K+\|A\|_{\mathrm{op}}}{2}\delta r_j
      +\frac{2 o(r_j^2)}{\delta r_j}.
\end{align*}
Taking the supremum over
\(|e|=1\), dividing by \(r_j\), and letting \(j\to\infty\), we obtain
\[
 \limsup_{j\to\infty}\frac{|p_j-p-Ah_j|}{r_j}
 \le \frac{K+\|A\|_{\mathrm{op}}}{2}\delta.
\]
Since \(\delta\) is arbitrary, \eqref{eq:supergradient-expansion}
follows.

\smallskip
Next, suppose that \(\phi\in C^1\) touches \(u\) from below at \(x\).
The convexity of the function\(\psi(z) =K|z|^2/2-u(z)\) near \(x\) gives
\(D^+u(x)\ne\varnothing\). Choose \(q\in D^+u(x)\). For all
sufficiently small \(h\),
\begin{equation} \label{phi-ext-oct-5}
 \phi(x+h)-\phi(x)
 \le u(x+h)-u(x)
 \le q\cdot h+\frac K2|h|^2.
\end{equation}
Applying these inequalities with \(h=te\) and \(h=-te\), divide by \(t\),
and let \(t\downarrow0\), we obtain \(q\cdot e=D\phi(x)\cdot e\) for every
unit vector \(e\). Hence, \(q=D\phi(x)\). Now, we substitute $q=D\phi(x)$ into \eqref{phi-ext-oct-5} and use the first-order extension of $\phi$ at $x$ to infer that
\[
 u(x+h)=u(x)+D\phi(x)\cdot h+o(|h|).
\]
This proves that $u$ is differentiable at $x$ and $Du(x) = D\phi(x)$.

\smallskip
Finally, assume the hypotheses of the last assertion in the lemma. For
\(\eta>0\),  we define
\[
 \psi_\eta(z)=P(z)-\eta|z-x|^2, \quad x \in \Omega.
\]
Since
\[
 u(z)-\psi_\eta(z)=R(z)+\eta|z-x|^2\ge0
\]
for \(z\) sufficiently close to \(x\), the function \(\psi_\eta\)
touches \(u\) from below at \(x\). Also,
\[
 D\psi_\eta(x)=p,\qquad D^2\psi_\eta(x)=A-2\eta I.
\]
The gradient \(p\) satisfies \(|p|>\gamma\) and \(p\notin\Sigma\). Therefore, 
\[
 \Pm(A-2\eta I)-\mu_0|p|\le c.
\]
Then, by letting \(\eta\downarrow0\), we obtain
\eqref{eq:semiconcave-classical-inequality}. The proof is complete.
\end{proof}

\smallskip
We also record the standard inf-convolution reduction. If \(v\) is bounded and lower semi-continuous in \(B_1\), define
\begin{equation} \label{inf-convol}
 v_\varepsilon(x)=\inf_{\xi \in B_1}\left\{v(\xi)+\frac{|x-\xi|^2}{2\varepsilon}\right\}.
\end{equation}
It is knwon that \(v_\varepsilon\) is semi-concave,   and \(v_\varepsilon\uparrow v\) as \(\varepsilon\downarrow0\), locally uniformly if \(v\) is continuous. 
\begin{lemma} \label{regularization-lemma} Assume that a lower semi-continuous function \(v \in L^\infty(B_2)\) satisfies
\begin{equation} \label{v-eqn-oct-7}
 \Mminus(D^2v,Dv)\le c \quad \text{in} \quad B_1,
\end{equation}
in the viscosity sense. Then, in the viscosity sense, \(v_\varepsilon\) satisfies the same inequality in every fixed smaller ball for all sufficiently small \(\varepsilon\) with the same threshold $\gamma$ and the same exceptional set \(\Sigma\). 
\end{lemma}
\begin{proof} For some some fixed sufficiently small   $\eta \in (0,1)$, let $\varepsilon \in (0,1)$ sufficiently small so that
\[
\varepsilon[\underset{B_1} \sup v - \underset{B_1} \inf v] \leq \frac{\eta^2}{8}.
\]
For every $x \in B_{1-\eta}$, let \(\phi \in C^2(B_\rho(x))\) with some small $\rho \in (0,\eta/2)$ that touches \(v_\varepsilon\) from below at $x$  with $D\phi(x) \not\in (\overline{B}_\gamma \cup \Sigma)$. Let  \(y \in B_1\) realizes the infimum in \eqref{inf-convol} at \(x\), i.e.
\[
v_\varepsilon(x) = v(y) + \frac{|x-y|^2}{2\varepsilon}.
\]
It is not hard to verify that $y \in \overline{B}_{1-\frac{\eta}{2}}$.
Then let the $C^2$ map $\psi : B_{\rho}(y) \rightarrow \R$ be defined by
\[
 z\longmapsto \psi(z) = \phi(x+z-y)-\frac{|x-y|^2}{2\varepsilon}.
\]
It follows that $\psi$ touches \(v\) from below at \(y\). Moreover,
\[
D\psi(y) = D\phi(x) \quad \text{and} \quad D^2\psi(y) = D^2\phi(x).
\]
Then $D\psi(y) \not\in (\overline{B}_\gamma \cup \Sigma)$, and it follows from \eqref{v-eqn-oct-7} that
\[
 \Mminus(D^2\phi(x),D\phi(x))=  \Mminus(D^2\psi(y),D\psi(y))\le c.
\]
Hence, we have verified that $v_\epsilon$ is the viscosity solution of
\[
\Mminus(D^2 v_\varepsilon,Dv_\varepsilon) \leq c \quad \text{in} \quad B_{1-\eta}
\]
and the proof is complete.
\end{proof}
\section{The cusp measure estimate} \label{Cusp-sec}
The following lemma is the main measure estimate. Its proof follows the cusp construction of \cite{ImbertSilvestre}. The new point is the elimination of the exceptional gradients in \(\Sigma\).

\begin{lemma}\label{lem:measure} For every given \(0< \lambda <\Lambda\) and \(\mu_0 \geq 0\), there are constants \(\gamma_0>0\), \(\delta_0\in(0,1)\), and \(M_0>1\) such that the following holds. Let \(\Sigma\) be a closed purely one-unrectifiable set, \(0\le\gamma\le\gamma_0\), and \(u:B_1\to\R\) be lower semi-continuous satisfying
\begin{equation} \label{M-eqn-Oct-3}
 \Mminus(D^2u,Du)\le1 \quad \text{in} \quad B_1,
\end{equation}
in the viscosity sense. Assume that 
\begin{equation} \label{hyp-oct-3}
 u\ge0, \quad \text{and} \qquad |\{u>M_0\}\cap B_1|>(1-\delta)|B_1|
\end{equation}
for some $\delta \in (0, \delta_0]$. Then \(u>1\) in \(B_{1/4}\).
\end{lemma}
\begin{proof}  We set
\[
M_0=2+5\sqrt2, \quad \gamma_0 = \sqrt{5}, \quad \text{and} \quad \delta_0 = \delta_0(n, \lambda, \Lambda, \mu_0) \in (0,1)
\]
where the number $\delta_0(n, \lambda, \Lambda, \mu_0)$ will be determined. We split the proof into two steps. In the first step, we prove the result with an additional assumption that $u$ is a semi-concave function. Then, in the second step, we remove such an additional assumption and prove the lemma.

\medskip
\noindent 
{\bf Step 1}: We prove Lemma~\ref{lem:measure} with an additional assumption that \(u\) is semi-concave. Assume that \eqref{hyp-oct-3} holds with some $\delta \in (0, \bar{\delta}_0]$ for some universal constant $\bar{\delta}_0 \in (0,1)$ depending only on $n, \lambda, \Lambda$ and $\mu_0$  that will be determined, and assume also that \(u\) is semi-concave. We prove the assertion of the lemma by a contradiction argument. Let
\[
\mathcal U=\{u>M_0\}\cap B_{1/4}, \quad \text{and} \quad \phi(z)=-10|z|^{1/2}, \quad z \in \R^n.
\]
By the inner regularity of Lebesgue measure, we can choose a compact set \(U\subset\mathcal U\) such that \(|\mathcal U\setminus U|<\delta|B_1|\). 

\smallskip
We assume by contradiction that there is some \(x_0\in B_{1/4}\) such that \(u(x_0)\le1\).  For every \(x\in U\), we minimize 
\[
 u(y)-\phi(y-x)=u(y)+10|y-x|^{1/2}
\]
over \(y\in B_1\). Note that as \(u\ge0\) and $x \in U \subset B_{1/4}$, every sequence \(z_j\in B_1\) with \(|z_j|\to1\) satisfies
\[
 \liminf_{j\to\infty}\bigl(u(z_j)+10|z_j-x|^{1/2}\bigr)\ge5\sqrt3>1+5\sqrt2\ge u(x_0)+10|x_0-x|^{1/2}.
\]
Thus a minimizing sequence stays in a compact subset of \(B_1\), and the lower semi-continuity of $u$ gives an existence of a minimizer \(y\in B_1\). In fact, choose \(\sigma_0\in(0,1/4)\) so small that
\[
10(3/4-\sigma_0)^{1/2}>1+5\sqrt2.
\]
If \(|y|\ge1-\sigma_0\), then as \(u(y)\ge0\) and \(|y-x|\ge3/4-\sigma_0\), we have
\[u(y)+10|y-x|^{1/2}> u(x_0)+10|x_0-x|^{1/2}\]
contradicting to the minimizing. Thus every minimizer belongs to \(B_{1-\sigma_0}\). Moreover, if $y$ is a minimizer, we have
\[
 u(y)\le1+5\sqrt2=M_0-1 < u(x),
\]
and therefore \(y\ne x\). Moreover, the function \(z\mapsto u(y)+\phi(z-x)-\phi(y-x)\) touches \(u\)
from below at \(y\); since \(y\ne x\), Lemma~\ref{lem:semiconcave-properties}
gives \(Du(y)=D\phi(y-x)\). In summary, for every $x \in U$ at the
corresponding contact point $y$ satisfies
\begin{equation} \label{conclusion-oct-2}
\left\{
\begin{array}{ll}
& y \in B_{1-\sigma_0}, \quad y\ne x, \quad u(y) \leq M_0-1, \quad Du(y)=D\phi(y-x),\\
& u(z)\ge u(y)+\phi(z-x)-\phi(y-x) \quad \text{for all } \ z \in B_1.
\end{array} \right.
\end{equation}
Let \(T\) be the set of all such contact points.  We note that $T$ is closed. Indeed, let $y_j\in T$ with $y_j\to y$. Choose $x_j\in U$ such that
$y_j$ minimizes $u(z)+10|z-x_j|^{1/2}$ over $z\in B_1$. Since $U$
is compact, after passing to a subsequence we have $x_j\to x\in U$.
For every $z\in B_1$,
\[
 u(y_j)+10|y_j-x_j|^{1/2}
 \leq u(z)+10|z-x_j|^{1/2}.
\]
By lower semicontinuity of $u$,
\[
 u(y)+10|y-x|^{1/2}
 \leq u(z)+10|z-x|^{1/2}.
\]
Thus $y\in T$, and $T$ is closed.

\smallskip
Now, since \(U\) is compact, \(T \subset \overline B_{1-\sigma_0}\) and it is closed, hence \(T\) is compact. Moreover, due to \eqref{conclusion-oct-2}
\[
 T\subset\{u\le M_0-1\}.
\]
Note also that  $ D\phi(z)=-5|z|^{-3/2}z$ and $D\phi$ is one-to-one on \(\R^n\setminus\{0\}\). Hence it follows from \eqref{conclusion-oct-2} that every \(y\in T\) determines the unique corresponding \(x\in U\). We write \(m(y)=x\). Thus \(m(T)=U\) and
\[
 Du(y)=D\phi(y-m(y)).
\]

\smallskip
Next, we show that \(Du\) is Lipschitz on \(T\). Since \(U\subset B_{1/4}\), \(T\subset B_{1-\sigma_0}\), \(u(x)>M_0\) on \(U\), and \(u(y)\le M_0-1\) on \(T\), we have
\[
u(m(y)) - u(y) \geq 1, \quad \forall  y \in T.
\] 
Also, due to Lemma~\ref{lem:semiconcave-properties}, $u$ is Lipschitz on \(\overline{B}_{1-\sigma_0}\). Hence, there is $\eta>0$, which depends on $u$, such that
\begin{equation} \label{y-my}
 |y-m(y)|\ge\eta>0,\qquad y\in T.
\end{equation}  
Let $y_1,y_2\in T$ with $y_1 \ne y_2$,  and let us denote $r=2|y_1-y_2|>0$. We consider the case that
\begin{equation} \label{case-1-oct-8}
 r<\min\{\sigma_0/2,\eta/2\}.
\end{equation}
In this case, we note that $B_{r}(y_j) \subset B_{1-\sigma_0/2}$ for $j=1,2$. Also, since
$y_j$ minimizes $u(\xi)-\phi(\xi-m(y_j))$ over $\xi \in B_1$,
\begin{equation} \label{minizer-oct-8}
 u(z)\geq u(y_j)
       +\phi(z-m(y_j))-\phi(y_j-m(y_j)), \quad \forall \ z \in B_{1}.
\end{equation}
Also, note that for any $z \in B_r(y_j)$, we have
\begin{equation} \label{w-est-oct-8}
 |(y_j-m(y_j))+t(z-y_j)|\geq |y_j-m(y_j)|-|(z-y_j)|\geq\eta-r>\eta/2, \quad \forall \ t \in [0, 1].
\end{equation}
Moreover, for $\xi\ne0$, we also note that
\begin{equation} \label{D2-phi-oct-8}
 D^2\phi(\xi)
 =-5|\xi|^{-3/2}I
  +\frac{15}{2}|\xi|^{-7/2}\xi\otimes\xi,
 \quad \text{and hence}\quad
 \|D^2\phi(\xi)\|_{\mathrm{op}}\leq5|\xi|^{-3/2}.
\end{equation}
Then, by using Taylor's formula, \eqref{w-est-oct-8},  and the fact that $D\phi(y_j-m(y_j))=Du(y_j)$, we see that
\[
\begin{aligned}
 \phi(w+h))-\phi(w)
 &=Du(y_j)\cdot h
   +\int_0^1(1-t)h^TD^2\phi(w+th)h\,dt\\
 &\geq Du(y_j)\cdot h
   -\frac52(2/\eta)^{3/2}|h|^2,
\end{aligned}
\]
where $w=y_j-m(y_j)$ and $h=z-y_j$, for $z \in B_{r}(y_j)$ and $j =1,2$. Since $|h|<r$, it follows this last estimate and \eqref{minizer-oct-8} that
\begin{equation} \label{u-z-below-oct-3}
 u(z)\geq u(y_j)+Du(y_j)\cdot(z-y_j)
          -\frac52(2/\eta)^{3/2}r^2, \quad \forall \ z \in B_r(y_j),\quad j=1,2.
\end{equation}
On the other hand, since $u$ is semi-concave in $B_1$, there is $K\geq0$ such that
$x\mapsto u(x)-\frac K2|x|^2$ is concave on
$B_{1-\sigma_0/2}$. Then, using 
the concavity of the map $x\mapsto u(x)-\frac K2|x|^2$ and $Du(y_2)=D\phi(y_2-m(y_2))$, we infer that
\[
u(z)-\frac K2|z|^2 \leq u(y_2)-\frac K2|y_2|^2
    +(Du(y_2)-Ky_2)\cdot(z-y_2), \quad z \in B_{1-\sigma_0/2}.
\]
Rearranging the terms and using 
$|z|^2-|y_2|^2-2y_2\cdot(z-y_2)=|z-y_2|^2$, we obtain
\[
 u(z)\leq u(y_2)+Du(y_2)\cdot(z-y_2)
       +\frac K2|z-y_2|^2, \quad z \in B_{1-\sigma_0/2}.
\]
Finally, as $|z-y_2|\leq |z-y_1|+|y_1-y_2|
          <r+\frac r2=\frac{3r}{2}$, it follows that
\[
 u(z)\leq u(y_2)+Du(y_2)\cdot(z-y_2)
          +\frac{9K}{8}r^2, \quad   z \in B_{1-\sigma_0/2}.
\]
On the other hand, observe also that \eqref{u-z-below-oct-3} particular implies that
\[
u(y_1)\geq u(y_2)+Du(y_2)\cdot(y_1-y_2)
          -\frac52(2/\eta)^{3/2}r^2.
\]
By combining the last two estimates, we obtain
\[
u(z) \leq u(y_1) + Du(y_2) \cdot (z-y_1) +Cr^2, \quad \forall \  z \in B_{1-\sigma_0/2}.
\]
From this, and by using \eqref{u-z-below-oct-3} with $j =1$ and the fact that $\overline{B}_{r}(y_1) \cup \overline{B}_r(y_2) \subset B_{1-\sigma_0/2}$, we obtain
\[
 (Du(y_1)-Du(y_2))\cdot(z-y_1)\le Cr^2, \quad \forall \ z \in \overline{B}_r(y_1).
\]
Taking \(z=y_1+re\) and \(z=y_1-re\), with \(e\in\mathbb S^{n-1}\), we obtain
\[
 |Du(y_1)-Du(y_2)|\le Cr
\]
when \eqref{case-1-oct-8} holds. On the other hand, if \eqref{case-1-oct-8} fails, i.e. \(r\ge\min\{\sigma_0/2,\eta/2\}\), it follows from \eqref{y-my} and the definition of $\phi$ that
\begin{align*}
|Du(y_1)-Du(y_2)| & = |D\phi(y_1 - m(y_1)) - D\phi(y_2 - m(y_2))| \leq 2 \sup_{y \in T}|D\phi (y -m(y))| \\
& \leq 2 \sup_{B_{5/4}\setminus B_{\eta}}|D\phi| \leq C |y_1 - y_2|.
\end{align*}
Hence, from both cases, we have proved that \(Du\) is Lipschitz on \(T\). 

\smallskip Next, note from \eqref{conclusion-oct-2} that for each $y\in T$,  $y \ne m(y)$ and 
\[
 Du(y)=D\phi(y-m(y))
 =-5|y-m(y)|^{-3/2}(y-m(y)).
\]
This, \eqref{y-my}, and $|y-m(y)|\le5/4$ imply that $2\sqrt5\le |Du(y)|\le5\eta^{-1/2}$ for all $y\in T$. Hence, we infer that the map $m$ is defined by
\[
 m(y)=y+ \frac{25Du(y)}{|Du(y)|^3}, \quad y \in T,
 \]
is also Lipschitz in $T$.

\smallskip
Next, from Lemma \ref{lem:semiconcave-properties} and \eqref{y-my}, for almost every \(y\in T\), the function \(u\) is twice differentiable and
\[
 D^2u(y)\ge D^2\phi(y-m(y)).
\]
Let us denote \(r_y=|y-m(y)|\). Then $r_y \in (\eta, 5/4]$ and \(D^2\phi(y-m(y))\) has the radial eigenvalue \((5/2)r_y^{-3/2}>0\) with eigenvector $(y -m(y))/r_y$ and the tangential eigenvalue \(-5r_y^{-3/2}\). Therefore
\[
 \operatorname{rank}D^2u(y)\ge1
\]
for almost every \(y\in T\). By Lemma~\ref{lem:semiconcave-properties},
\(Du|_T\) is differentiable relative to \(T\) with derivative
\(D^2u(y)\) at almost every such point. Lemma~\ref{lem:contact-gradient} yields
\[
 |\{y\in T:Du(y)\in\Sigma\}|=0.
\]
As \(r_y\le5/4\), and we have
\[
 |Du(y)| = |D\phi(y-m(y))|=5r_y^{-1/2}\ge2\sqrt5 > \gamma_0.
\]
Thus, for almost every \(y\in T\), Lemma~\ref{lem:semiconcave-properties}
gives a second-order expansion of \(u\), and
Lemma~\ref{lem:contact-gradient} gives \(Du(y)\notin\Sigma\).
At such a point, we also have \(|Du(y)|\ge2\sqrt5>\gamma\). From this and since \(u\) satisfies
\eqref{M-eqn-Oct-3}, we apply Lemma~\ref{lem:semiconcave-properties}to infer that
\[
 \Pm(D^2u(y))-\mu_0|Du(y)|\le1
\]
or equivalently, 
\begin{equation} \label{u-test-PDE}
 \lambda\tr(D^2u(y))^+-\Lambda\tr(D^2u(y))^--\mu_0|Du(y)|\le1.
\end{equation}
In addition, since \(D^2u (y)\ge D^2\phi(y-m(y))\), we have
\[
 \tr(D^2u(y))^-\le\tr\bigl((D^2\phi(y-m(y)))^-\bigr) = 5(n-1)r_y^{-3/2}\le C r_y^{-3/2}.
\]
Using this, $|Du(y)|=5r_y^{-1/2}$ and $r_y \in (\eta, 5/4]$, we infer from \eqref{u-test-PDE} that
\[
\begin{aligned}
 \lambda\operatorname{tr}\bigl((D^2u(y))^+\bigr)
 &\leq1+\Lambda\operatorname{tr}\bigl((D^2u(y))^-\bigr)
          +\mu_0|Du(y)|\\
 &\leq1+
 \left(5\Lambda(n-1)+\frac{25}{4}\mu_0\right)r_y^{-3/2}.
\end{aligned}
\]
Therefore,
\begin{equation} \label{Du-est-Oc-3}
 |D^2u(y)|
 \leq\operatorname{tr}\bigl((D^2u(y))^+\bigr)
     +\operatorname{tr}\bigl((D^2u(y))^-\bigr)
 \leq C\left(1+r_y^{-3/2}\right)
\end{equation}
with some $C = C(n, \lambda, \Lambda, \mu_0)>0$.

\smallskip
Now, by Kirszbraun's theorem (see \cite{Kirszbraun, Valentine}), we can extend $m:T\to\R^n$ to a Lipschitz map
$\widetilde m:\R^n\to\R^n$. Due to Rademacher's theorem (see \cite{Rademacher} and \cite[Theorem 3.2, p. 103]{Evans}), $\widetilde m$ is differentiable at almost every where in $\R^n$. Let $y$ be a density point of $T$ at
which $\widetilde m$ is differentiable and the relative expansion of
$Du|_T$ from Lemma~\ref{lem:semiconcave-properties} holds. For $z\in T$ with
$z$ sufficiently close to $y$,
\begin{equation} \label{Du-expansion-oct-8}
 Du(z)-Du(y)=D^2u(y)(z-y)+o(|z-y|).
\end{equation}
Since $\widetilde m=m$ on $T$,
\[
\begin{aligned}
 z-m(z)-(y-m(y))
 &=(z-y)-\bigl(\widetilde m(z)-\widetilde m(y)\bigr)\\
 &=(I-D\widetilde m(y))(z-y)+o(|z-y|).
\end{aligned}
\]
Using $Du(z)=D\phi(z-m(z))$ on $T$ and expanding $D\phi$ at
$y-m(y)$, we obtain
\[
 Du(z)-Du(y)
 =D^2\phi(y-m(y))(I-D\widetilde m(y))(z-y)
  +o(|z-y|).
\]
By Comparing this expansion and the expansion \eqref{Du-expansion-oct-8}, we infer that
\[
 \bigl[D^2u(y)
   -D^2\phi(y-m(y))(I-D\widetilde m(y))\bigr](z-y)
 =o(|z-y|),
 \qquad z\in T,\quad z\to y.
\]
For each $e\in\mathbb S^{n-1}$, as in \eqref{claim-oct-5}, it follows from the density-one property of $T$ at
$y$ that there is $z_h\in T$ with $z_h-y=he+o(h)$ as $h\downarrow0$.
Substituting $z=z_h$, dividing by $h$, and letting $h\downarrow0$, we get
\[
 D^2u(y)=D^2\phi(y-m(y))(I-D\widetilde m(y)).
\]
From this identity, \eqref{Du-est-Oc-3}, and the fact that  \(\|(D^2\phi(y-m(y)))^{-1}\|\le Cr_y^{3/2}\) due to the explicit formulation of $D^2\phi$ in \eqref{D2-phi-oct-8}, we infer that
\[
 |D\widetilde m(y)|\le1+C r_y^{3/2}|D^2u(y)|\le C_0
\]
for almost every \(y\in T\), where $C_0 = C_0 (n, \lambda, \Lambda, \mu_0)>0$. Since $\widetilde m(T)=m(T)=U$, every $x\in U$ has at least one
preimage in $T$, it follows from the area formula \cite[Theorem 3.8, p.119]{Evans} that
\[
 |U|
 =\int_U 1\,dx
 \leq\int_{\R^n}
   \#\{y\in T:\widetilde m(y)=x\}\,dx
 =\int_T|\det D\widetilde m(y)|\,dy \leq C_0^n|T|.
\]
This and as $|U|\ge|B_{1/4}|-2\delta|B_1|$, it follows that
\[
|B_{1/4}|-2\delta|B_1| \leq C_0^n|T|.
\]
On the other hand, we have \(T\subset\{u\leq M_0-1\}\cap B_1
\subset\{u\leq M_0\}\cap B_1
\). Then, it follows from the measure assumption in \eqref{hyp-oct-3} that  \(|T|\le\delta|B_1|\), and therefore
\[
|B_{1/4}|-2\delta|B_1| \leq  C_0 ^n\delta|B_1|.
\]
Note that $C_0 = C_0(n, \lambda, \Lambda, \mu_0)>0$. Then, we obtain a contradiction for  $\delta \in (0, \bar{\delta}_0]$ when choosing \(\bar{\delta}_0 = \delta_0(n, \lambda, \Lambda, \mu_0) \in (0,1)\) sufficiently small.\\

\noindent
\smallskip
{\bf Step 2}: We remove the  semi-concavity assumption on \(u\) used in {\bf Step 1} by using inf-convolution. Let $M_0>0, \mu_0\ge0$ and $\bar{\delta}_0 = \bar{\delta}_0(n, \lambda, \Lambda, \mu_0) \in (0,1)$ be the numbers defined in {\bf Step 1}. We choose 
\[ \delta_0 = \bar{\delta}_0/4.\]  
Let \(u\) be lower semi-continuous and we assume that \eqref{hyp-oct-3} holds with some $\delta \in (0, \delta_0]$. Let us denote \(v=\min\{u,2M_0\}\). We note that \(v\) is also satisfies  \eqref{M-eqn-Oct-3} in the viscosity sense. Let \(v_\varepsilon\) be the inf-convolution defined in \eqref{inf-convol} in Section~\ref{semi-concave}. For every fixed small \(\eta \in (0,1)\), it follows from Lemma \ref{regularization-lemma} that, in  in \(B_{1-\eta}\), \(v_\varepsilon\) is a semi-concave and it satisfies in the viscosity sense the same equation as $u$ for all sufficiently small \(\varepsilon\). Moreover, observe that
\[
 \{u>M_0\}=\underset{\varepsilon>0}\bigcup\{v_\varepsilon>M_0\}
\]
and the set $\{v_\varepsilon>M_0\}$ increases as \(\varepsilon\downarrow0\). Hence $|\{v_\varepsilon>M_0\}|$ converges to \(|\{u>M_0\}|\) as \(\varepsilon\downarrow0\). 

\smallskip
Fix \(x\in B_{1/4}\), and we aim to prove that $u(x)>1$. Choose \(\eta \in (0, 1- 4|x|)\) so that \(|x|<(1-\eta)/4\).  Then, we define the dilated function
\[
 \bar v_\varepsilon(z)=v_\varepsilon(a z),
 \qquad z\in B_1
\]
where $a= 1-\eta$. As $a \in (0,1)$,  in viscosity sense, $\bar{v}_\epsilon$ satisfies
\[
 \mathcal P^-_{\lambda, \Lambda}(D^2\bar{v}_\epsilon)- a\mu_0|D\bar{v}_\epsilon|
 \leq a^2 \quad \text{in} \quad B_1 \quad \text{when} \quad D\bar{v}_\epsilon \not\in \big(\overline B_{a\gamma} \cup a\Sigma\big).
\]
Because $a \in (0,1)$, in viscosity sense, we also have
\begin{equation} \label{v-epsilont-eqn}
 \mathcal P^-_{\lambda, \Lambda}(D^2\bar{v}_\epsilon)- \mu_0|D\bar{v}_\epsilon|
 \leq 1, \quad \text{in} \quad B_1 \quad \text{when} \quad D\bar{v}_\epsilon \not\in \big(\overline B_{a\gamma} \cup a\Sigma\big).
\end{equation}
Note that the set $a\Sigma$ is closed and purely one-unrectifiable as proved in Lemma \ref{lem:stability-sigma}.  As $\eta>0$ is sufficiently small,  we have
$a^n>1/2$. Since $v_\varepsilon\uparrow v=\min\{u,2M_0\}$ in $B_a$
and $\{v\leq M_0\}=\{u\leq M_0\}$, it follows from the continuity of Lebesgue measure and  the measure hypothesis \eqref{hyp-oct-3}  that
\[
 |\{v_\varepsilon\leq M_0\}\cap B_a|
 \longrightarrow |\{u\leq M_0\}\cap B_a|
 \leq\delta|B_1| \quad \text{as} \quad \epsilon \longrightarrow 0^+.
\]
Hence, for all sufficiently small $\varepsilon$,
\[
 |\{v_\varepsilon\leq M_0\}\cap B_a|
 <2\delta_0|B_1|
 <\bar\delta_0|B_a|.
\]
From this, and the change of variables $x=az$, we see that
\[
\begin{aligned}
 |\{\bar v_\varepsilon>M_0\}\cap B_1|
 &=a^{-n}|\{v_\varepsilon>M_0\}\cap B_a|\\
 &=a^{-n}\bigl(|B_a|-|\{v_\varepsilon\leq M_0\}\cap B_a|\bigr)\\
 &>(1-\bar\delta_0)|B_1|.
\end{aligned}
\]
Thus $\bar v_\varepsilon$ satisfies the measure hypothesis  \eqref{hyp-oct-3} when $\epsilon$ is sufficiently small.  From this and as \eqref{v-epsilont-eqn}, we apply the result that we just proved in {\bf Step 1} for $\bar{v}_\epsilon$ with sufficiently small $\epsilon$ to obtain
\[
 \bar v_\varepsilon>1\quad\text{in }B_{1/4},
 \qquad\text{hence}\qquad
 v_\varepsilon>1\quad\text{in }B_{(1-\eta)/4}.
\]
Since \(u\ge v_\varepsilon\) and $x \in B_{(1-\eta)/4}$, we obtain \(u(x)>1\). As \(x\in B_{1/4}\) was arbitrary, the proof is complete.
\end{proof}

\section{A strict barrier and the doubling property} \label{doubling-property}

The radial barrier in \cite{ImbertSilvestre} may have gradients in $\Sigma$. The following technical comparison lemma accounts for these gradients.
\begin{lemma}\label{lem:strict-comparison}
Let $\Sigma\subset\R^n$ be closed and purely one-unrectifiable. Let $\Omega\subset\R^n$ be non-empty, open and bounded, and $u\in C(\overline\Omega)$ satisfy
\[
 \Mminus(D^2u,Du)\le1\quad\text{in }\Omega
\]
in the viscosity sense. Suppose $\psi\in C^2(\overline\Omega)$ such that $u\ge\psi$ on $\partial\Omega$ and that there are $\eta,c_0>0$ such that, for every $x\in\Omega$ and $|q|\le\eta$,
\[
 |D\psi(x)+q|>\gamma,
 \quad \text{and} \quad
 \Pm(D^2\psi(x))-\mu_0|D\psi(x)+q|\ge1+c_0.
\]
Assume also that $D^2\psi(x)$ has a positive eigenvalue for every $x\in\Omega$. Then $u\ge\psi$ in $\Omega$.
\end{lemma}
\begin{proof} Suppose by contradiction that $\min_{\overline\Omega}(u-\psi)=-a<0$ with some $a>0$. Since $u-\psi\ge0$ on $\partial\Omega$, we can choose a compact set $K\Subset\Omega$ such that
\[
 \min_K(u-\psi)=-a,
 \qquad
 u-\psi>-\frac a4\quad\text{on }\overline\Omega\setminus K.
\]
Let $u_\varepsilon$ be the inf-convolution of $u$ over $\overline\Omega$ defined in \eqref{inf-convol}. It follows that $u_\epsilon$ converges uniformly to $u$ on $\overline\Omega$. Also, for a sufficiently small $\varepsilon$, and let $K_\epsilon$ be a neighborhood of $K$, $K \Subset K_\epsilon \Subset \Omega$, such that $u_\varepsilon$ satisfies the same viscosity inequality as $u$ in $K_\epsilon$, and $u_\varepsilon$ is semi-concave in $K_\epsilon$. Let us fix \(\varepsilon>0\) sufficiently small such that these properties hold,  and 
\[
 \min_Kw<-\frac{3a}{4},
 \quad \text{and} \quad
 w>-\frac a3\quad\text{on }\overline\Omega\setminus K
\]
where $w=u_\varepsilon-\psi$. Decrease $\eta$, if necessary, we may assume without loss of generality that  $2\eta\operatorname{diam}\Omega<a/6$. For each $q\in\overline B_\eta$, let $y_q$ be any minimizer of $w(x)-q\cdot x$ over $\overline\Omega$. Comparing its values at a point where $w<-3a/4$ and at any point of $\overline\Omega\setminus K$, we infer that every such minimizer $y_q$ belongs to $K$. Let us define
\[
 T=\displaystyle{\underset{q\in\overline B_\eta}\cup}\big \{y\in\overline\Omega:  w(y)-q\cdot y=\min_{x \in \overline\Omega}(w(x)-q\cdot x)\big\} \subset K.
\]
Because contact relation is closed in $K\times\overline B_\eta$, $T$ is compact. Corresponding to some $q \in \overline{B}_\eta$, at $y_q \in T$, the function $x\mapsto\psi(x)+q\cdot(x-y_q)+u_\varepsilon(y_q)-\psi(y_q)$ touches $u_\varepsilon$ from below. Then, applying Lemma~\ref{lem:semiconcave-properties}, we have
\[
 Du_\varepsilon(y_q)=D\psi(y_q)+q \quad \text{or} \quad Du_\varepsilon(y_q)- D\psi(y_q) =q.
\]
Thus, $q$ is uniquely determined, and the map 
\[
Q : T \mapsto \overline{B}_\eta \subset R^n: \quad Q =Du_\varepsilon-D\psi
\]
is well-defined and $Q(T) = \overline{B}_\eta$. We claim that the map $Q$ is Lipschitz on $T$. Indeed, let $y_j\in T$ with  $j=1,2$, we infer from the minimum at $y_1$ and the minimum at $y_2$ that
\begin{equation} \label{minimize-oct-4}
 \begin{aligned}
 w(z)&\ge w(y_1)+Q(y_1) \cdot(z-y_1), \quad z \in \Omega,\\
 w(y_1)&\ge w(y_2)+Q(y_2) \cdot(y_1-y_2).
 \end{aligned}
\end{equation}
Let $h = \operatorname{dist}(K_\epsilon,\partial\Omega)/2)>0$. Then, choose a semi-concavity constant $\kappa_0>0$ for $w$ on $\overline{K}_\epsilon$. At $y_2$, we have
\[
 w(z)\le w(y_2)+ Q(y_2) \cdot(z-y_2)+\frac{\kappa_0}{2}|z-y_2|^2, \quad \forall \ z\in B_h(y_2).
\]
Combining this with \eqref{minimize-oct-4}, we infer that
\begin{equation} \label{Q-y12-oct-8}
 (Q(y_1) -Q(y_2))\cdot(z-y_1)\le\frac{\kappa_0}{2}|z-y_2|^2, \quad \forall \ z\in B_h(y_2).
\end{equation}
Let us denote $d=|y_1-y_2|$. If $2d<h$, we use $z=y_1\pm de$ with $|e|=1$ in \eqref{Q-y12-oct-8} to infer that
\[
 \pm(Q(y_1)-Q(y_2))\cdot e\le2\kappa_0 d \quad \text{and therefore} \quad
 \quad |Q(y_1)-Q(y_2)|\le2 \kappa_0|y_1-y_2|.
\]
On the other hand, if $2d\ge h$, then $|Q(y_1)-Q(y_2)|\le2\eta\le(4\eta/h)d$. Thus $Q$ is Lipschitz on $T$ with the Lipschitz constant
\[
L = \max\{2\kappa_0, 4\eta/h\}.
\]
Since $Q(T)=\overline B_\eta$ and $Q$ Lipschitz, it follows that
\[
|B_\eta|= |Q(T)| \leq L^n |T|.
\]
Therefore, $|T|>0$.

\smallskip
As $u_\varepsilon$ is semi-concave in $K_\epsilon$ and $T \subset K \subset K_\epsilon$, it follows from Alexandrov's theorem that at almost every $y\in T$, the second-order expansion of $u_\varepsilon$ exists. Moreover, as $x \mapsto u_\varepsilon(x)-\psi(x)-q\cdot x$ attains the  minimum at $y$, we have
\[
 D^2u_\varepsilon(y)\ge D^2\psi(y).
\]
Hence $D^2u_\varepsilon(y)$ has a positive eigenvalue and $\operatorname{rank}D^2u_\varepsilon(y)\ge1$. Since $Du_\varepsilon|_T=Q+D\psi$ is Lipschitz, its relative derivative at almost every such $y$ is $D^2u_\varepsilon(y)$ by Lemma~\ref{lem:semiconcave-properties}. Then, we infer from Lemma~\ref{lem:contact-gradient} that
\[
 |\{y\in T:Du_\varepsilon(y)\in\Sigma\}|=0.
\]
From this and as $|T|>0$, we can choose $y \in T$ where $u_\epsilon$ admits the second-order expansion  and $Du_\varepsilon(y)\notin\Sigma$. From assumption on $\psi$, $|Du_\varepsilon(y)|=|D\psi(y)+q|>\gamma$. It then follows from Lemma~\ref{lem:semiconcave-properties}, the assumption in the lemma, and the fact that $D^2u_\varepsilon(y)\ge D^2\psi(y)$ that
\[
 1\ge\Pm(D^2u_\varepsilon(y))-\mu_0|Du_\varepsilon(y)|
 \ge\Pm(D^2\psi(y))-\mu_0|D\psi(y)+q|
 \ge1+c_0.
\]
Hence, we have a contradiction as $c_0>0$, and the proof is complete.
\end{proof}

\smallskip
We now state and prove the doubling property.
\begin{lemma}\label{lem:doubling} For every  $0 <\lambda <\Lambda$, $\mu_0 \geq 0$,  and $\gamma_0>0$ ,  there is $M_1>1$ such that the following assertion holds. Assume that $\Sigma$ is a closed purely one-unrectifiable set,  $0\le\gamma\le\gamma_0$, and $u\in C(B_2)$ is non-negative and satisfies
\[
\Mminus(D^2u,Du)\le1\quad\text{in }B_2,
\]
in the viscosity sense,  and $u>M_1$ in $B_{1/4}$. Then
\begin{equation} \label{doubling-est}
u>1\quad\text{in }B_{3/2}.
\end{equation}
\end{lemma}

\begin{proof} Let  us fixed $p>0$ such that $\lambda(p+1)-\Lambda(n-1)-2\mu_0>0$. Then, let $b(x)=|x|^{-p}$, and it follows from the direct calculation that
\[
 \Pm(D^2b)-\mu_0|Db|
 =pr^{-p-2}\bigl(\lambda(p+1)-\Lambda(n-1)-\mu_0r\bigr), \quad r = |x| <2.
\]
Now, with some $M_1>0$ to be determined shortly, let
\[
 B(x)=\frac{M_1}{2\cdot4^p}\bigl(|x|^{-p}-(7/4)^{-p}\bigr),
 \qquad \frac14\le|x|\le\frac74.
\]
We have $B=0$ on $\partial B_{7/4}$ and $B<M_1/2$ on $\partial B_{1/4}$. For some fixed $\gamma_0>0$, by choosing $M_1$ sufficiently large, we obtain, for $1/4<|x|<7/4$ and $|q|\le1$,
\[
 B\ge1\quad\text{if }|x|\le3/2,
 \qquad |DB|>\gamma_0+1,
 \qquad \Pm(D^2B)-\mu_0|DB+q|\ge2.
\]
Indeed, the last inequality follows from
\[
 \Pm(D^2B)-\mu_0|DB+q|
 \ge \frac{M_1p}{2\cdot4^p}r^{-p-2}
       \bigl(\lambda(p+1)-\Lambda(n-1)-\mu_0r\bigr)-\mu_0.
\]
Note also that the radial eigenvalue of $D^2B$ is positive.

\smallskip
Now, we apply Lemma~\ref{lem:strict-comparison} in $B_{7/4}\setminus\overline B_{1/4}$. Note that in this domain, on the outer boundary $\partial B_{7/4}$, we have$u\ge0=B$, and on the inner boundary $\partial B_{1/4}$, we have $u\ge M_1>B$ due to the assumption in the lemma and the continuity of $u$. Thus $u\ge B$ in the annulus $B_{7/4}\setminus\overline B_{1/4}$, which is almost the assertion in the lemma except that we may have the equality. To derive the assertion with the strict inequality, we note that for with $\theta>0$ sufficiently small, we also have $u \geq (1+\theta)B$ on $\partial (B_{7/4}\setminus\overline B_{1/4})$. Moreover,
\[
 \begin{aligned}
& \Pm(D^2((1+\theta)B))-\mu_0|D((1+\theta)B)+q|\\
 &=(1+\theta)\Big[\Pm(D^2B)-\mu_0\big|DB+\frac q{1+\theta}\big|\Big]\\
 &\ge2(1+\theta),\qquad |q|\le1.
 \end{aligned}
\]
Moreover, $|D((1+\theta)B)+q|>\gamma$.  Then, applying Lemma~\ref{lem:strict-comparison} again, we also infer that $u\ge(1+\theta)B>1$ in $B_{3/2}\setminus\overline B_{1/4}$. From this and the assumption that $u>M_1>1$ in $\overline{B}_{1/4}$, we see that $u>1$ on $B_{3/2}$. Hence, the estimate \eqref{doubling-est} is proved, and the proof is complete.
\end{proof}

\smallskip
Now, we combine Lemma~\ref{lem:measure} and Lemma \ref{lem:doubling} to derive the following result.
\begin{corollary}\label{cor:measure-point}For every $0 <\lambda<\Lambda$ and $\mu_0 \geq 0$, there are $\gamma_0>0$, $\delta_0 \in(0,1)$, and $M>1$ such that the following holds. Let $\Sigma$ be closed and purely one-unrectifiable, $0\le\gamma\le\gamma_0$, and $u\in C(B_2)$ be non-negative and satisfy
\[
 \Mminus(D^2u,Du)\le1\quad\text{in }B_2
\]
in the viscosity sense. If
\[
 |\{u>M\}\cap B_1|>(1-\delta)|B_1|,
\]
for some $\delta \in (0, \delta_0]$, then $u>1$ in $B_{3/2}$.
\end{corollary}

\begin{proof}
Let $M_0,\delta_0,\gamma_0$ be the constants in Lemma~\ref{lem:measure}, then let $M_1$ be the constant defined in Lemma~\ref{lem:doubling}. Let $M=M_0M_1$, and $v=u/M_1$. Note that $v$ satisfies, in the viscosity sense, that
\[
 \Pm(D^2v)-\mu_0|Dv|\le\frac1{M_1}\le1 \quad \text{in} \quad B_2  \quad \text{when} \quad |Dv|>\gamma/M_1 \quad \text{and} \quad Dv\notin M_1^{-1}\Sigma.
\]
The set $M_1^{-1}\Sigma$ is closed and purely one-unrectifiable by Lemma~\ref{lem:stability-sigma}. Since
\[
 |\{v>M_0\}\cap B_1|>(1-\delta)|B_1|,
\]
It follows from Lemma~\ref{lem:measure} that $u>M_1$ in $B_{1/4}$. From this, we apply Lemma~\ref{lem:doubling} to obtain $u>1$ in $B_{3/2}$.
\end{proof}

\smallskip
By using the dilation and scaling, we record the following result which is the scaled version of Corollary \ref{cor:measure-point}.
\begin{corollary}\label{cor:scaled-measure}
Let $\gamma_0,\delta_0, M$ be as in Corollary~\ref{cor:measure-point}. Suppose $0<r\le2$, $\kappa\ge1$, $0\le\gamma\le\gamma_0$, $\Sigma$ is closed and purely one-unrectifiable. Suppose that $u\in C(B_r)$ is non-negative satisfies, in the viscosity sense,
\[
\Mminus(D^2u,Du)\le\kappa\quad\text{in }B_r.
\]
If
\begin{equation} \label{measure-hyp-oct-8-1}
 |\{u>\kappa M\}\cap B_{r/2}|>(1-\delta)|B_{r/2}|,
\end{equation}
with $\delta \in (0, \delta_0]$, then $u>\kappa$ in $B_{r/2}$.
\end{corollary}
\begin{proof} Using the scaling and dilation, let us denote $v(x)=\kappa^{-1}u(rx/2)$ for $x\in B_2$. The threshold and exceptional set for $v$ are $r\gamma/(2\kappa)\le\gamma_0$ and $(r/(2\kappa))\Sigma$, respectively. At active gradients,
\[
 \begin{aligned}
 \Pm(D^2v(x))-\mu_0|Dv(x)|
 &\le\Pm(D^2v(x))-\frac r2\mu_0|Dv(x)|\\
 &=\frac{r^2}{4\kappa}\bigl(\Pm(D^2u(rx/2))-\mu_0|Du(rx/2)|\bigr)
 \le\frac{r^2}{4}\le1.
 \end{aligned}
\]
 The same calculation applies to test functions. It also follows from the measure hypothesis \eqref{measure-hyp-oct-8-1} that 
\[
 |\{v>M\}\cap B_1|>(1-\delta)|B_1|.
\]
Hence, by Corollary~\ref{cor:measure-point}, it follows that  $v>1$ in $B_{3/2}$. From this, we infer that $u>\kappa$ in $B_{r/2}$ and the corollary is proved.
\end{proof}

\section{The \texorpdfstring{\(L^\varepsilon\)}{L-epsilon}-estimates} \label{L-eps-Sec}
We begin with the following growing-ink-spots lemma introduced by N. V. Krylov and M. V. Safonov, see \cite{Safonov} and also \cite[Lemma B.1, p. 409]{Krylov} for the parabolic case.
\begin{lemma} \label{lem:ink}
Let $E\subset F\subset B_1$ be open, and let $\delta\in(0,1)$. Suppose $|E|\le(1-\delta)|B_1|$ and every ball $B\subset B_1$ satisfying $|B\cap E|>(1-\delta)|B|$ is contained in $F$. Then
\[
 |E|\le(1-5^{-n}\delta)|F|.
\]
\end{lemma}

\begin{proof} We provide the proof for completeness. Note the if $F=\emptyset$ or $F=B_1$, then the assertion of the lemma is trivial. Otherwise, for a fixed  $x\in F$, we choose $\overline B_\rho(x)\subset F$. For $0\le t\le1$, let us define the nested balls
\[
 \tilde{B}_t=B_{(1-t)\rho+t}((1-t)x).
\]
Observe that $\tilde{B}_0=B_\rho(x)$, and $\tilde{B}_1=B_1$. We denote
\[
 t_x=\sup\{t\in[0,1]:\tilde{B}_t\subset F\}.
\]
Then $0<t_x<1$, $\tilde{B}_{t_x}=\underset{t<t_x}\bigcup \tilde{B}_t\subset F$, and $x\in \tilde{B}_{t_x}$. For $t>t_x$ the ball $\tilde{B}_t$ is not contained in $F$. From this, the density assumption, and the continuity of $t\mapsto|\tilde{B}_t\cap E|/|\tilde{B}_t|$, it follows that
\begin{equation} \label{ball-measur-oct-7}
 |\tilde{B}_{t_x}\cap E|\le(1-\delta)|\tilde{B}_{t_x}|.
\end{equation}
As $F \subset \underset{x \in F}\bigcup\tilde{B}_{t_x}$, by Vitali covering lemma, there is a sub-collection $\{B_j\}_{j \in \mathbb{N}}$ of disjoint balls from $\{\tilde{B}_{t_x}\}_{x \in F}$ such that $F\subset\underset{j}\bigcup (5B_j)$. Since $B_j\subset F$, it follows from \eqref{ball-measur-oct-7} that
\[
 |F\setminus E|
 \ge\sum_j|B_j\setminus E|
 \ge\delta\sum_j|B_j|
 \ge5^{-n}\delta|F|.
\]
This implies the assertion in the lemma, and the proof is complete.
\end{proof}

\smallskip
From Corollary~\ref{cor:scaled-measure} and a standard iteration method based on Lemma~\ref{lem:ink} as in \cite{Krylov, Safonov},  we can state and prove the following $L^\varepsilon$-estimate theorem. For completeness, we provide the proof.
\begin{theorem}[$L^\varepsilon$-estimate theorem]\label{thm:Leps} For every $0< \lambda <\Lambda$ and $\mu_0 \geq 0$, there are $\gamma_0>0$, $\varepsilon>0$, and $C>0$ such that the following holds. Suppose that $\Sigma$ is closed and purely one-unrectifiable, $0\le\gamma\le\gamma_0$, and $u\in C(B_2)$ is non-negative and satisfies, 
\[
\Mminus(D^2u,Du)\le1\quad\text{in }B_2
\]
in the viscosity sense, and $\displaystyle{\inf_{B_1}u\le1}$. Then
\begin{equation} \label{L-espilon-Oct5}
 |\{u>t\}\cap B_1|\le Ct^{-\varepsilon},\qquad t>0.
\end{equation}
\end{theorem}
\begin{proof} Let $M, \delta_0$ be the constants defined in Corollary~\ref{cor:measure-point}. Let $\delta = \delta_0$, and 
\[
 A_k=\{u>M^k\}\cap B_1,\qquad k\ge1.
\]
Note that if $|A_1|>(1-\delta)|B_1|$, then it folllows from Corollary~\ref{cor:measure-point} that $u>1$ in $B_{3/2}$. From this and as $u \in C(\overline B_1)$, we see that that $\underset{B_1}\inf  u>1$, and we obtain a contradiction. Hence
\begin{equation} \label{A-1-oct-5}
 |A_1|\le(1-\delta)|B_1|.
\end{equation}
\smallskip
Next, we claim that
\begin{equation} \label{A-k-oct-6}
 |A_{k+1}|\le(1-5^{-n}\delta)|A_k|, \quad \forall \ k \geq 1.
\end{equation}
To prove the claim, we apply Lemma~\ref{lem:ink} with $E=A_{k+1}$ and $F=A_k$. We then need to verify the assumptions in Lemma~\ref{lem:ink}. Since $E \subset A_1$,  the first assumption in  Lemma~\ref{lem:ink}, $|E| \leq (1-\delta)|B_1|$, holds due to \eqref{A-1-oct-5}. Now, suppose that $B_r(x)\subset B_1$ satisfies
\[
 |A_{k+1}\cap B_r(x)|>(1-\delta)|B_r|.
\]
Since $B_{2r}(x)\subset B_2$ and $M^k\ge1$, we apply Corollary~\ref{cor:scaled-measure} to $u$ in $B_{2r}(x)$, with $\kappa=M^k$ to conclude that $u>M^k$ in $B_r(x)$, and therefore $B_r(x)\subset A_k$, and this verifies second assumption in Lemma~\ref{lem:ink}. Thus, \eqref{A-k-oct-6} follows from Lemma~\ref{lem:ink}.

\smallskip
We take $\vartheta=1-5^{-n}\delta$ and $\varepsilon=-\log\vartheta/\log M>0$, and we iterate \eqref{A-k-oct-6} with the fact that $1-\delta < 1 -5^{-n} \delta =\vartheta$ to yield
\begin{equation} \label{A-k-decay-est}
 |A_k|\le |B_1|\vartheta^k=|B_1|M^{-\varepsilon k},\qquad k\ge1.
\end{equation}
If $t\ge M$, we choose $k\ge1$ such that $M^k\le t<M^{k+1}$ and then apply \eqref{A-k-oct-6} to infer that
\[
 |\{u>t\}\cap B_1|\le|A_k|
 \le |B_1| M^{-\varepsilon k} \leq |B_1|M^\varepsilon t^{-\varepsilon}.
\]
On the other hand, when $0<t \leq M$,  it is trivial that
\[ |\{u>t\}\cap B_1|\le|B_1|  \leq |B_1| M^\epsilon t^{-\epsilon}.
\]
In other words, we have proved that
 \[ |\{u>t\}\cap B_1|\le|B_1|  \leq |B_1| M^\epsilon t^{-\epsilon}, \quad \forall \ t>0
\]
which infers \eqref{L-espilon-Oct5} and the proof is complete.
\end{proof}
\section{Proofs of the main results} \label{Proof-Sec}
We now have all ingredients to prove the main results of the paper. We begin with the proof of Theorem~\ref{thm:harnack-intro}.
\begin{proof}[Proof of Theorem~\ref{thm:harnack-intro}] Because the argument requires some technical room, we begin by establishing a Harnack inequality for the equation that is posed in $B_2$. Let $c\ge0$, and suppose that $u\in C(B_2)$ is non-negative and satisfies, in the viscosity sense,
\begin{equation} \label{eqn-oct-5}
 \Mminus(D^2u,Du)\le c,\qquad \Mplus(D^2u,Du)\ge-c \quad \text{in} \quad B_2,
\end{equation}
where $\Sigma$ is closed and purely one-unrectifiable. We claim that
\begin{equation}\label{harnack-local-B2}
 \sup_{B_{1/2}}u\le C\left(\inf_{B_{1/2}}u+c+\gamma\right),
\end{equation}
where $C=C(n,\lambda,\Lambda,\mu_0)$.
For a small $\zeta>0$, let us denote
\[
 A=\max\left\{c+\inf_{B_{1/2}}u,\frac\gamma{\gamma_0}\right\}+\zeta,
\]
where $\gamma_0$ is chosen to work in Corollary~\ref{cor:measure-point} and Theorem~\ref{thm:Leps}. By using the scaling $u \mapsto u/A$ and replace $c$ by $c/A\le1$, the threshold $\gamma$ by $\gamma/A\le\gamma_0$, and the set $\Sigma$ by $A^{-1}\Sigma$, it suffices to prove the bound $\underset{B_{1/2}} \sup u$ by a universal constant under the assumption that $\underset{B_{1/2}} \inf  u \leq 1$.

\smallskip
Let \(\beta>0\) to be determined, and 
\begin{equation} \label{t-def-oct-5}
 t=\max_{\overline B_1}u(x)(1-|x|)^\beta,
\end{equation}
where the product is extended by zero on \(\partial B_1\).  Note that as \(u\) is continuous in \(B_2\), the maximum in the definition above is attained. Then, we define
\[
 h_t(x)=t(1-|x|)^{-\beta},\qquad x\in B_1.
\]
It is clear from the definition that \(h_t\ge u\) in \(B_1\). Also, note that if \(t\le1\), the bound $\underset{B_{1/2}} \sup u$ is immediate. Hence, we only need to consider the case that \(t>1\). In this case, the maximum in \eqref{t-def-oct-5} is attained at some \(x_0\in B_1\), and we denote
\[
 H_0 = u(x_0)=h_t(x_0) =t(2r)^{-\beta}\ge1\quad \text{for} \quad r=\frac{1-|x_0|}{2} \in (0, 1/2].
\]
Since $\underset{B_1}\inf u\le1$ and $\gamma\le\gamma_0$, it follows from  Theorem~\ref{thm:Leps} that
\begin{equation} \label{measure-1-oct-5}
 |\{u>H_0/2\}\cap B_r(x_0)|
 \le 2^\varepsilon C_0 H_0^{-\varepsilon}
 =2^\varepsilon C_0 t^{-\varepsilon}(2r)^{\beta\varepsilon},
\end{equation}
where $C_0>0$ and $\varepsilon>0$ depend only on $n,\lambda,\Lambda,\mu_0$.

\smallskip
Next, let $M,\delta_0$ be as in Corollary~\ref{cor:measure-point}, $\delta = \delta_0$, and fix
\[
 \beta=\max\left\{1,\frac n\varepsilon\right\}.
\]
Then, we choose $\mu\in(0,1/2)$ sufficiently small, depending only on $M$ and $\beta$,
such that
\[
 M\bigl((1-\mu)^{-\beta}-1\bigr)\le\frac14.
\]
Since $\beta\ge1$ and $(1-\mu)^{-\beta}-1\ge\beta\mu\ge\mu$, we obtain
\[
 \frac{\mu^2}{(1-\mu)^{-\beta}-1}\le\mu\le1,
 \quad \text{and} \quad
 \frac{\mu\gamma}{(1-\mu)^{-\beta}-1}
 \le\gamma\le\gamma_0.
\]
On the other hand,  recall that
\[
 u(x)\le h_t(x)=t(1-|x|)^{-\beta},\qquad
 2r=1-|x_0|, \quad \text{and} \quad H_0=t(2r)^{-\beta}.
\]
Then, for every $x\in B_{2\mu r}(x_0)$, we have
\[
 1-|x|
 \ge1-|x_0|-|x-x_0|
 >2r-2\mu r
 =2r(1-\mu).
\]
Therefore,
\[
 u(x)
 \le t(1-|x|)^{-\beta}
 \le t\bigl(2r(1-\mu)\bigr)^{-\beta}
 =(1-\mu)^{-\beta}H_0,
 \qquad x\in B_{2\mu r}(x_0).
\]
Now, let us define $A_\mu = (1-\mu)^{-\beta} -1$, and
\[
 v(x)=\frac{(1-\mu)^{-\beta} H_0-u(x_0+\mu r x)}{A_\mu H_0},\qquad x\in B_2.
\]
Then \(v\ge0\), \(v(0)=1\), and the corresponding transformed exceptional set for $v$ is \(-\frac{\mu r}{A_\mu H_0}\Sigma\), hence purely one-unrectifiable. Moreover, at active gradients,
\begin{align*}
& \Pm(D^2v)-\mu_0|Dv|\le \Pm(D^2v)-\mu r\mu_0|Dv| \\
 & =-\frac{(\mu r)^2}{A_\mu H_0}\bigl(\Pp(D^2u)+\mu_0|Du|\bigr)
 \le\frac{(\mu r)^2}{A_\mu H_0}\le1.
\end{align*}
The corresponding transformed threshold for $v$ is
\[
 \frac{\mu r\gamma}{A_\mu H_0}\le\frac{\mu\gamma}{A_\mu}\le\gamma_0.
\]
Since \(v(0)=1\), it follows from the contrapositive of Corollary~\ref{cor:measure-point} that
\begin{equation} \label{v-density-oct-6}
 |\{v\le M\}\cap B_1|\ge\delta|B_1|.
\end{equation}
For every $x \in B_1$ such that \(v(x) \le M\), we have
\[
 u(x_0+\mu r x)
 \ge H_0((1-\mu)^{-\beta} -MA_\mu)
 =H_0\bigl(1-(M-1)A_\mu\bigr)
 \ge\frac{3H_0}{4}>\frac{H_0}{2}.
\]
Hence, we infer from \eqref{v-density-oct-6} that
\begin{equation} \label{measure-2-oct-5}
 |\{u>H_0/2\}\cap B_{\mu r}(x_0)|\ge\delta|B_1|(\mu r)^n.
\end{equation}
Since $2r\le1$ and $\beta\varepsilon\ge n$, it follows from the two measure bounds \eqref{measure-1-oct-5} and \eqref{measure-2-oct-5} that
\[
 \delta|B_1|\mu^n r^n
 \le 2^\varepsilon C_0 t^{-\varepsilon}(2r)^{\beta\varepsilon}
 \le 2^{n+\varepsilon}C_0t^{-\varepsilon}r^n.
\]
Hence
\[
 t\le\left(\frac{2^{n+\varepsilon}C_0}{\delta|B_1|\mu^n}\right)^{1/\varepsilon}.
\]
Since $u\le h_t$ in $B_1$, we have 
\[ \sup_{B_{1/2}}u\le2^\beta\max\{1,t\}\le C,\]
where $C=C(n,\lambda,\Lambda,\mu_0)$. Note that when undoing the normalization, and letting $\zeta\downarrow0$, we obtain
\[
 \sup_{B_{1/2}}u
 \le C\max\left\{c+\inf_{B_{1/2}}u,\frac\gamma{\gamma_0}\right\}
 \le C\left(c+\inf_{B_{1/2}}u+\gamma\right),
\]
for $C = C(n, \lambda, \Lambda, \mu_0)>0$. This proves \eqref{harnack-local-B2}.

\smallskip
We now assume that $u$ satisfies the hypotheses of Theorem~\ref{thm:harnack-intro} in $B_1$. Let us denote
\[
 C_F=\|f\|_{L^\infty(B_1)}+\|F(0,0,\cdot)\|_{L^\infty(B_1)}.
\]
By Proposition~\ref{prop:extremal-reduction}, in the viscosity sense, $u$ satisfies
\[
 \Mminus(D^2u,Du)\le C_F,\qquad
 \Mplus(D^2u,Du)\ge-C_F\quad\text{in }B_1.
\]
Fix $z\in B_{1/2}$ and define $U_z(y)=u(z+y/4)$ for $y\in B_2$. Since $B_{1/2}(z)\subset B_1$, this function is well defined. Its corresponding threshold is $\gamma/4$, its exceptional set is $\Sigma/4$, and its drift coefficient is $\mu_0/4$. Its extremal inequalities have right-hand sides $C_F/16$ and $-C_F/16$. Increasing the drift coefficient to $\mu_0$ preserves both inequalities. Thus the estimate \eqref{harnack-local-B2} implies that
\begin{equation}\label{harnack-small-ball}
 \sup_{B_{1/8}(z)}u
 \le C_*\left(\inf_{B_{1/8}(z)}u+C_F+\gamma\right),
\end{equation}
where $C_*=C_*(n,\lambda,\Lambda,\mu_0)\ge1$ is independent of $z$.

\smallskip
Now, let $x,y\in B_{1/2}$ and 
\[
 z_i=y+\frac{i}{8}(x-y),\qquad i=0,\ldots,8.
\]
Then $z_i\in B_{1/2}$ and $|z_{i+1}-z_i|=|x-y|/8<1/8$. Applying \eqref{harnack-small-ball} with center $z_i$ gives
\[
 u(z_{i+1})\le C_*\bigl(u(z_i)+C_F+\gamma\bigr),
 \qquad i=0,\ldots,7.
\]
Iterating these eight inequalities, we obtain
\[
 u(x)\le C_*^8u(y)
       +(C_F+\gamma)\sum_{i=1}^{8}C_*^i
 \le C\bigl(u(y)+C_F+\gamma\bigr).
\]
Taking the supremum over $x\in B_{1/2}$ and the infimum over $y\in B_{1/2}$, we obtain the desired Harnack inequality in the theorem.
\end{proof}
From Theorem~\ref{thm:harnack-intro}, the proof of  Theorem~\ref{thm:holder-intro} is standard and it follows the known steps as in \cite{Krylov, Safonov}. However, there are several various non-standard manipulations related to $\gamma$. We provide the proof for completeness.
\begin{proof}[Proof of Theorem~\ref{thm:holder-intro}]
By Proposition~\ref{prop:extremal-reduction}, it follows that, in the viscosity sense, $u$ satisfies
\begin{equation} \label{Pucci-Oct-5}
 \Mminus(D^2u,Du)\le C_F, \quad \text{and}\quad \Mplus(D^2u,Du)\ge-C_F \quad \text{in} \quad B_1
\end{equation}
where
\[
C_F= \|f\|_{L^\infty(B_1)}+\|F(0,0,\cdot)\|_{L^\infty(B_1)}.
\]
We first derive a scale-invariant oscillation estimate. Fix $x_0\in B_{1/2}$ and $0<R\le1/2$, and write
\[
 M_R=\sup_{B_R(x_0)}u,\qquad m_R=\inf_{B_R(x_0)}u,\qquad \omega(R)=M_R-m_R.
\]
Put $r=R/2$ and we define in $B_{2r}(x_0)=B_R(x_0)$ the following two functions
\[
 v=u-m_R+r\gamma,\qquad w=M_R-u+r\gamma.
\]
The function $v$ is a viscosity solution for the same equation as $u$. On the other hand, $w$ is a viscosity solution of
\[
F^*(D^2w, Dw,x) = f^*(x) \quad \text{in} \quad B_R(x_0) \quad \text{when} \quad Dw \not \in (\bar{B}_\gamma \cup (-\Sigma))
\]
where
\[
 F^*(X,p,x)=-F(-X,-p,x), \quad \text{and} \quad f^*(x)=-f(x), \quad (X, p, x) \in \mathbb{S}^n \times \R^n \times B_1.
\]
The operator $F^*$ satisfies \eqref{F-cond} as $F$, and $-\Sigma$ is purely one-unrectifiable.

\smallskip
Apply Theorem~\ref{thm:harnack-intro} to $v$ and $w$ in $B_r(x_0)$, after the scaling $x \mapsto x_0+ry$ with $y\in B_1$. The resulting Harnack estimates hold in $B_{r/2}(x_0)=B_{R/4}(x_0)$. We note that for $v$, the scaled operator is
\[
 F_r(X,p,y)=r^2F(r^{-2}X,r^{-1}p,x_0+ry),
\]
so its ellipticity constants are $\lambda,\Lambda$, its gradient coefficient is $r\mu_0\le\mu_0$, its corresponding threshold is $r\gamma$, and its corresponding right-hand-side bound in \eqref{Pucci-Oct-5} is $r^2C_F$. The same statements hold for the reflected operator associated with $w$. Since
\[
 \inf_{B_{r/2}(x_0)}v\ge r\gamma \quad \text{and} \quad \inf_{B_{r/2}(x_0)}w\ge r\gamma,
\]
the additive term $r\gamma$ in Theorem~\ref{thm:harnack-intro} is at most either infimum. Hence the two Harnack inequalities hold with a constant $C=C(n,\lambda,\Lambda,\mu_0)>1$ such that
\[
 M_{R/4}-m_R+r\gamma\le C\bigl(m_{R/4}-m_R+r\gamma+r^2C_F\bigr),
\]
and
\[
 M_R-m_{R/4}+r\gamma\le C\bigl(M_R-M_{R/4}+r\gamma+r^2C_F\bigr).
\]
By adding these two inequalities, we obtain
\begin{equation*} 
 \omega(R)+\omega(R/4)+2r\gamma\le C\bigl(\omega(R)-\omega(R/4)+2r\gamma+2r^2C_F\bigr).
\end{equation*}
Therefore,
\begin{equation} \label{osc-R-oct-5}
 \omega(R/4)\le\theta\omega(R)+C_0(R\gamma+R^2C_F), \quad \text{with} \quad \theta=\frac{C-1}{C+1}<1,
\end{equation}
where $C_0$ depends only on $n,\lambda,\Lambda,\mu_0$.

\smallskip
Choose $\alpha\in(0,1)$ so small that $\theta<4^{-\alpha}$. Iterating the inequality \eqref{osc-R-oct-5} on the radii $R_j=4^{-j}R_0$, we infer that
\begin{equation} \label{omega-R-j-oct-5}
 \omega(R_j)\le\theta^j\omega(R_0)+C_0\sum_{i=0}^{j-1}\theta^{j-1-i}\bigl(R_i\gamma+R_i^2C_F\bigr).
\end{equation}
Because $R_i=4^{-i}R_0$ and $4^\alpha\theta<1$, we have
\[
 \omega(R_j)\le C4^{-\alpha j}\bigl(\omega(R_0)+R_0\gamma+R_0^2C_F\bigr).
\]
Now, recall that $K=\|u\|_{L^\infty(B_1)}+C_F$. Note that $K=0$, then $u=0$ and the result is immediate. Hence, it remains to consider $K>0$. Let
\[
 R_0 =
\left\{
\begin{array}{ll} 
\frac14\min\left\{1,\frac K\gamma\right\} & \quad \text{if} \quad \gamma>0, \\
\frac14 & \quad \text{if} \quad \gamma =0.
\end{array} \right.
\]
Then $B_{R_0}(x_0)\subset B_1$ and
\[
 \omega(R_0)+R_0\gamma+R_0^2C_F\le CK.
\]
From this and \eqref{omega-R-j-oct-5},  it follows that
\[
 \omega(R_j)
 \le C4^{-\alpha j}
       \bigl(\omega(R_0)+R_0\gamma+R_0^2C_F\bigr)
 \le CK\left(\frac{R_j}{R_0}\right)^\alpha.
\]
Consequently, for $0<s\le R_0$, we choose $j$ so that $R_{j+1}<s\le R_j$ and obtain
\begin{equation} \label{omega-s-oct-5}
 \operatorname{osc}_{B_s(x_0)}u\le CK\left(\frac{s}{R_0}\right)^\alpha.
\end{equation}

\smallskip
Now, let $x,y\in B_{1/2}$ and $d=|x-y|$. If $d\le R_0/4$, we apply \eqref{omega-s-oct-5} for the ball $B_{4d}(x)$. Since $y\in B_{4d}(x)$, we infer that
\[
 |u(x)-u(y)|\le \operatorname{osc}_{B_{4d}(x)}u\le CK R_0^{-\alpha}d^\alpha.
\]
On the other hand, if $d>R_0/4$, we also have
\[
 |u(x)-u(y)|\le2\|u\|_{L^\infty(B_1)}\le CK R_0^{-\alpha}d^\alpha.
\]
Hence,
\[
|u(x)-u(y)|\le CK R_0^{-\alpha}d^\alpha, \quad \forall \ x, y \in B_{1/2},
\]
where $C= C(n, \lambda, \Lambda, \mu_0)>0$. Finally, from the choice of $R_0$, we have $R_0^{-\alpha}\le C(1+(\gamma/K)^\alpha)$, and therefore
\[
 [u]_{C^\alpha(B_{1/2})}
 \le CKR_0^{-\alpha}
 \le C\left(K+K^{1-\alpha}\gamma^\alpha\right).
\]
The proof is then complete.
\end{proof}
\section{Examples and counterexamples} \label{Example-Sec}
We conclude the paper with examples showing both the size of the admissible class $\Sigma$ and the necessity of excluding rectifiable gradient channels. 
We note that the counterexamples do not assert that pure one-unrectifiability is necessary for every possible exceptional set. They show that one-dimensional rectifiable channels are obstructions and, when $n\ge2$, that a condition based only on \(n\)-dimensional measure or density cannot suffice. Pure one-unrectifiability is the geometric condition used in our adaptation of the cusp and barrier contact arguments of \cite{ImbertSilvestre}.

\smallskip
The following three examples provide the admissible sets $\Sigma$ that our main theorems are applicable.
\begin{example}  Let \(\Sigma=\{p_j:j\ge1\}\) with a given sequence $\{p_j\}_j \in \R^n$ satisfying \(|p_j|\to\infty\). Then \(\Hh^1(\Sigma)=0\) and \(\Sigma\) is purely one-unrectifiable. The equation is therefore not needed at an arbitrary unbounded discrete family of large gradients, and Theorems~\ref{thm:holder-intro} and \ref{thm:harnack-intro} apply.
\end{example}

\begin{example} If \(\Sigma\) is closed and \(\dim_{\Hh}\Sigma<1\), then \(\Hh^1(\Sigma)=0\) and $\Sigma$ is  purely one-unrectifiable and  Theorems~\ref{thm:holder-intro} and \ref{thm:harnack-intro} apply. Therefore the admissible class $\Sigma$ includes unbounded Cantor-type sets of Hausdorff dimension strictly smaller than one.
\end{example}

\begin{example}  Assume \(n\ge2\), let \(\mathcal K\subset\R^2\) be the four-corner Cantor set defined in \eqref{four-corner-Cantor}. We embed $\mathcal K$ in \(\R^n\), and let
\[
 \Sigma=\bigcup_{j=1}^\infty(2^je_1+\mathcal K),
\]
where $e_1 = (1, 0, 0, \ldots, 0) \in \R^n$.  The union is locally finite, hence closed. Note also that $\Hh^1(\mathcal K) \in (0, \infty)$ and $\mathcal K$ is purely one-unrectifiable as indicated in \cite[Example 15.2 and the discussion following
Definition 15.3, pp.~203-204]{Mattila} Each translate of $\mathcal K$ is purely one-unrectifiable, and it follows from Lemma~\ref{lem:stability-sigma} that \(\Sigma\) is purely one-unrectifiable. Moreover, every translate of $\mathcal K$ has positive finite \(\Hh^1\)-measure. Thus the admissible class of exceptional set can be unbounded and Hausdorff one-dimensional.
\end{example}

\smallskip
In the following two counterexamples, one may fix $F(X,p,x)=\tr X$, $f=0$, and $\gamma=0$. At a contact with a smooth function $u$, every test function has gradient $Du$, so no viscosity condition is imposed when $Du\in\Sigma$. 

\begin{example} \label{ex:line}
Let \(\Sigma=\R e_1\). For every smooth function depending only on \(x_1\), the gradient lies in \(\Sigma\), hence they are viscosity solutions to considered class of fully nonlinear equation and both extremal inequalities. In particular, for each $k \in \mathbb{N}$, let 
\[
 u_k(x)=\sin(kx_1), \quad x = (x_1, x') \in \R \times \R^{n-1}, \quad |x| <1.
\]
We have
\[
Du_k(x)=k\cos(kx_1)e_1 \in \Sigma,
\qquad
\|Du_k\|_{L^\infty(B_1)}=k.
\]
Also, for \(\alpha\in(0,1)\), since \(|u_k|\le1\), for all \(x,y\in B_{1/2}\),
\[
|u_k(x)-u_k(y)|
\le \min\{k|x-y|,2\}
\le 2^{1-\alpha}k^\alpha|x-y|^\alpha.
\]
Consequently,
\[
[u_k]_{C^\alpha(B_{1/2})}
\le 2^{1-\alpha}k^\alpha.
\]
On the other hand, for \(k>\pi\), we choose
\[
x_k=\frac{\pi}{2k}e_1,
\quad \text{and} \quad
y_k=-\frac{\pi}{2k}e_1.
\]
Both points belong to \(B_{1/2}\), and moreover
\[
u_k(x_k)=1,
\quad
u_k(y_k)=-1,
\quad \text{and} \quad
|x_k-y_k|=\frac{\pi}{k},
\]
and then
\[
[u_k]_{C^\alpha(B_{1/2})}
\ge
\frac{|u_k(x_k)-u_k(y_k)|}{|x_k-y_k|^\alpha}
=
\frac{2}{\pi^\alpha}k^\alpha.
\]
Thus, for \(k>\pi\),
\[
\frac{2}{\pi^\alpha}k^\alpha
\le [u_k]_{C^\alpha(B_{1/2})}
\le 2^{1-\alpha}k^\alpha.
\]
In particular, the H\"older seminorm grows at order \(k^\alpha\),
and hence tends to infinity as \(k\to\infty\), despite
\(\|u_k\|_{L^\infty(B_1)}\le1\). As such, the H\"{o}lder estimate in Theorem \ref{thm:holder-intro} does not hold.

\smallskip
On the other hand, for each $k\in \mathbb{N}$, let
\[
 v_k(x)=1+kx_1^2, \quad x = (x_1, x') \in \R \times \R^{n-1}, \quad |x| <1.
\]
We see that $v_k(x) \geq 0$ in $B_1$, and
\[
 \inf_{B_{1/2}}v_k=1 \quad 
 \qquad
 \sup_{B_{1/2}}v_k=1+\frac k4\longrightarrow\infty.
\]
Thus the Harnack inequality in Theorem \ref{thm:harnack-intro}  does not hold.
\end{example}

\begin{example} \label{ex:ray} Let \(\Sigma=\{te_1:t\ge0\}\), and for each $k \in \mathbb{N}$, we define
\[
u_k(x)=\frac{1+\tanh(kx_1)}{2},
\qquad x=(x_1,x')\in \R \times \R^{n-1}: \ |x| <1.
\]
Then, we have
\[
Du_k(x)
=\frac{k}{2}\operatorname{sech}^2(kx_1)e_1
=\frac{k}{2\cosh^2(kx_1)}e_1\in\Sigma
\]
and $u_k$ is a viscosity solution of \eqref{intro-PDE}.  Moreover, \(0<u_k<1\) in \(B_1\). In addition, for every \(\alpha\in(0,1)\), 
and for \(k>2\), choose \(x_k=k^{-1}e_1\) and
\(y_k=-k^{-1}e_1\), which belong to \(B_{1/2}\).
Then
\[
|u_k(x_k)-u_k(y_k)|=\tanh(1),
\quad \text{and} \quad
|x_k-y_k|=\frac{2}{k}.
\]
Therefore,
\[
[u_k]_{C^\alpha(B_{1/2})}
\ge
\frac{|u_k(x_k)-u_k(y_k)|}{|x_k-y_k|^\alpha} =
\frac{\tanh(1)}{2^\alpha}k^\alpha.
\]
On the other hand, note that \(\|Du_k\|_{L^\infty(B_1)}=k/2\) and
\(0<u_k<1\). Then, it follows that for all \(x,y\in B_{1/2}\),
\[
|u_k(x)-u_k(y)|
\le \min\left\{\frac{k}{2}|x-y|,1\right\}
\le \left(\frac{k}{2}|x-y|\right)^\alpha.
\]
Thus, for \(k>2\),
\[
\frac{\tanh(1)}{2^\alpha}k^\alpha
\le [u_k]_{C^\alpha(B_{1/2})}
\le \frac{1}{2^\alpha}k^\alpha.
\]
Hence, the H\"{o}lder seminorm grows at order \(k^\alpha\), and estimate in Theorem \ref{thm:holder-intro} does not hold for $u_k$.
 
\smallskip
Similarly, for each $k \in \mathbb{N}$, let us define
\[
 v_k(x)=1+e^{kx_1} \quad x = (x_1, x') \in \R \times \R^{n-1}: \ |x| <1.
\]
Again \(Dv_k\in\Sigma\) everywhere and $v_k$ is a viscosity solution to \eqref{intro-PDE}. Also,
\[
 \inf_{B_{1/2}}v_k=1+e^{-k/2}\to1,
 \qquad
 \sup_{B_{1/2}}v_k=1+e^{k/2}\to\infty, \quad \text{as} \quad k \to \infty.
\]
Hence, there is no uniform Harnack inequality for $v_k$, and Theorem \ref{thm:harnack-intro} does not hold.
\end{example}
\section*{Acknowledgement}
The research of T. Phan was partially supported by Simons Foundation, grant \# 769369.  The author would like to thank professor Sungwon Cho (Gwangju National University of Education) and Dr. Junyuan Fang (Louisiana State University) for the thorough and stimulating discussions on the results and methods developed in \cite{ImbertSilvestre, Krylov, Mooney}. 

\end{document}